\documentclass{amsart}[11pt,fullpage]

\usepackage[margin=1.25in]{geometry}

\usepackage{amsmath,amssymb,amsthm}
\usepackage{graphicx}
\usepackage[colorlinks=true,linkcolor=blue,citecolor=blue]{hyperref}

\usepackage{tikz}

\usepackage{mycommands}

\newcommand{\SL}{\text{SL}}
\newcommand{\GL}{\text{GL}}
\newcommand{\SU}{\text{SU}}

\newcommand{\tr}{\text{tr}}
\newcommand{\X}{\mathfrak X}
\newcommand{\Y}{\mathfrak Y}
\newcommand{\Z}{\mathfrak Z}

\newcommand{\bi}{{\bf i}}
\newcommand{\nl}{\natural}
\newcommand{\aug}{\text{Aug}}
\newcommand{\dc}[1]{\ensuremath{\Sigma(#1)}}

\newcommand{\hc}{\ensuremath{HC_0^{ab}}}

\title[Character varieties via the cord ring]{Character varieties of knot complements and branched double-covers via the cord ring}

\author{Christopher R.\ Cornwell}
\address{CIRGET \\ Universit\'e du Qu\'ebec \`a Montr\'eal}
\email{cornwell@cirget.ca}

\begin{document}


\begin{abstract}We study the relationship between Ng's abelian cord ring and $\SL_2\bb C$ characters of the two-fold branched cover $\dc K$. Augmentations, and their corresponding rank, play a central role in the relationship. Our study also leads to a correspondence between trace-free $\SL_2\bb C$ characters of a knot complement and augmentations. Whether the trace-free character is metabelian or not is determined by the rank of the corresponding augmentation. Necessary and sufficient conditions are given on the augmentation for it to correspond to an $SU(2)$ character. This provides an effective tool for finding representations of relevance to singular instanton homology. We also give necessary and sufficient conditions for an augmentation to induce a non-abelian $\SL_2\bb R$ character of $\dc K$ and exploit this to discuss branched covers with left-orderable fundamental group.
\end{abstract}

\maketitle

\section{Introduction}
\label{sec-introduction}

For a given compact connected 3-manifold $M$ we denote the $\SL_2\bb C$ character variety of $\pi_1(M)$ by $X(M)$. Let $\dc K$ be the branched double-cover over a knot $K\subset S^3$ and let $E_K$ be the complement of a tubular neighborhood of $K$. Here we study $X(\dc K)$ and the meridian-traceless characters $X_{TF}(E_K)\subset X(E_K)$ by exploiting a relationship to the abelian cord ring \cite{Ng05b}, which is related to the knot contact homology of $K$.

Knot contact homology $HC_\ast(K)$ is a non-commutative algebra associated to $K$ which arises out of symplectic field theory\footnote{For the experts, we are only considering the version of knot contact homology with $Q=1$, i.e.\ as an algebra over $\bb Z[H_1(\Lambda_K)]$, where $\Lambda_K$ is the conormal torus over $K$.}, and its zero-graded part $HC_0(K)$ admits a description, as the \textit{cord algebra}, that solely utilizes the topology of the pair $(S^3,K)$. We are interested in an application arising from the cord algebra description which relates $HC_0(K)$ to $X(\dc K)$. More precisely, the complexification of the \textit{abelian cord ring} $\hc(K)$, which is a commutative quotient of $HC_0(K)$, admits a homomorphism $\Z:\hc(K)\otimes\bb C \to \bb C[X(\dc K)]$ which is an isomorphism if $K$ is 2-bridge \cite{Ng05b}. We will extend this result to a more general context.

Call a homomorphism $\ve:\hc(K)\to\bb C$ a \textit{reflective augmentation} of $K$. A key piece of input for what follows is the notion of the \textit{rank} of a (reflective) augmentation coming out of the author's previous work (see Section \ref{SecBGC_K}). Write $\mathfrak r^{ab}(K)$ for the set of ranks of reflective augmentations of $K$; note, $\mathfrak r^{ab}(K)$ is finite for each $K$ {--} in fact, bounded above by the bridge number \cite{Cor14}. For any $r>0$ there exist knots with $r = \text{max}\ \mathfrak r^{ab}(K)$, e.g.\ $(r,s)$-torus knots with $r<s$ (see also \cite[Section 5]{Cor14b}). 

There is a natural involution $\tau$ on $\dc K$ which is of interest to us. We focus on the set of $\tau$-equivariant characters in $X(\dc K)$, denoted $X^\tau(\dc K)$, which is algebraic (see Section \ref{sec-dcK}). Note that 3-bridge knots (among others) satisfy the condition of the following theorem.

\begin{thm}[see Theorem \ref{thm-equivariant}] Suppose that $\text{max}\ \mathfrak r^{ab}(K)\le 3$. Then $X^\tau(\dc K) = X(\dc K)$ and $(\hc(K)\otimes\bb C)/\sqrt0 \cong \bb C[X(\dc K)]$.
\label{thm-allequivariant-isomorphism}
\end{thm} 

Our proof of Theorem \ref{thm-allequivariant-isomorphism} passes through a correspondence between reflective augmentations (of rank at most 3) and trace-free characters in $X_{TF}(E_K)$ (see Theorem \ref{thm-varietycorrespondence}). The correspondence is such that we obtain the following.

\begin{thm}[see Corollary \ref{thm-IrrMetabelian}]Reflective augmentations of $K$ with rank two (resp.\ three) are in bijective (resp.\ 2-1) correspondence with irreducible metabelian (resp.\ trace-free non-metabelian) characters in $X(E_K)$.
\label{cor-IrrMetabelian}
\end{thm}

The map giving the isomorphism of Theorem \ref{thm-allequivariant-isomorphism} is not generally an isomorphism. In particular, it fails to be if there exist reflective augmentations of rank larger than four (Proposition \ref{prop-NontrivialKernel}). 



We also provide applications in two directions. First we describe an explicit $\SL_2\bb C$ representation of $\pi_1(E_K)$ from a rank three augmentation (Section \ref{sec-unitary}), wherein augmentation values appear as traces. This gives conditions on the augmentation necessary for the representation to be conjugate into $SU(2)$. We find that they are also sufficient. 

\begin{thm}[see Theorem \ref{thm-unitary2}] Suppose that $K$ has an $n$-crossing diagram and let $\ve$ be a reflective augmentation of $K$ of rank three. Then a trace-free non-binary-dihedral $SU(2)$ character corresponds to $\ve$ if and only if an explicit set of $\binom n2 + \binom n3$ values of $\ve$ are in the interval $[-2,2]\in\bb R$.
\label{thm-unitary}
\end{thm}

Calling an augmentation \textit{elliptic} if it satisfies the above bounded condition, if no elliptic augmentation exists then $K$ has only binary-dihedral, trace-free representations into $SU(2)$. The singular instanton homology $I^\nl(K)$ for such knots is simple. Following \cite{Zentner}, call such a knot $\SU(2)$-simple.  We find all alternating knots with at most 10 crossings which are $\SU(2)$-simple. We also note that every non-alternating knot with crossing number at most 11 does have $\SU(2)$ representations which are not binary-dihedral, providing some computational evidence that $SU(2)$-simple knots are alternating (see \cite{Zentner}).

Our second application relates to the orderability of the fundamental group of $\dc K$. Given a rank three augmentation that is not elliptic we find a non-abelian $\SL_2\bb R$ representation of $\pi_1(\dc K)$. 

\begin{thm}[see Theorem \ref{thm-SL2R}] Suppose $\ve$ is a real-valued rank three reflective augmentation of $K$ that is not elliptic. Then there is a non-abelian representation $\rho\co\pi_1(\dc K)\to\text{SL}_2\bb R$.
\label{thm-sl2R-orderable}
\end{thm}
Such representations can be used to study left-invariant orders on $\pi_1(\dc K)$, provided their Euler class vanishes. We discuss some examples in Section \ref{sec-unitary} and a way to generate families of homology spheres with left orderable fundamental group (see Corollaries \ref{cor-det1covers} and \ref{cor-braidsatellites}).

The paper is organized as follows. In Sections \ref{SecBGCharVar}, \ref{SecBGC_K}, and \ref{sec-dcK} we review background on ($\SL_2\bb C$) character varieties, the abelian cord ring, and results concerning $X(\dc K)$ and $X_{TF}(E_K)$ respectively. We also prove a few needed results in these sections, Proposition \ref{prop-augtotrace} in Section \ref{sec-HCaugmentations} being of particular importance. In Section \ref{sec-results} we prove Theorem \ref{thm-allequivariant-isomorphism} and Theorem \ref{cor-IrrMetabelian}. Section \ref{sec-unitary} contains our two applications, Theorems \ref{thm-unitary} and \ref{thm-sl2R-orderable} and some examples. We include an appendix tabulating numbers of reflective augmentations of small knots.

\subsection*{Acknowledgments}The author is grateful for support from a CIRGET postdoc and from the Institut Mittag-Leffler. He would like to thank Steven Boyer for discussing with him the relevance of the results to left-invariant orderings. He would also like to thank Anh Tran for directing him to Nagasato's results and Lenny Ng for comments on an earlier draft. 

\section{Character varieties}
\label{SecBGCharVar}
Throughout the paper we write $\left[ A_{ij} \right]$ to denote a matrix which has $A_{ij}$ in row $i$ and column $j$. When needed we will also use subscripts, e.g.\ $\left[ A_{ij}\right]_{1\le i,j\le m}$, to indicate the size of the matrix. The determinant will be written either $\left| A_{ij}\right|$ or $\det\left[ A_{ij}\right]$.

Let $G$ be a finitely generated group. Upon choosing a generating set $\{g_1,\ldots,g_r\}$ of $G$ a representation $\rho\in\text{Hom}(G,\SL_2\bb C)$ is determined by $(\rho(g_1),\ldots,\rho(g_r))\in \bb C^{4r}$. Define the representation space $R(G) = \{(\rho(g_1),\ldots,\rho(g_r)) | \rho\in\text{Hom}(G,\SL_2\bb C)\}$, which is an algebraic set, well-defined up to canonical isomorphism. Write $X(G)$ for the set of associated characters $\chi_\rho\co G\to\bb C$, defined by $\chi_\rho(g) = \tr(\rho(g))$ (where $\tr$ is the trace and $\rho\in\text{Hom}(G,\SL_2\bb C)$). If $M$ is a compact connected 3-manifold then we write $X(M)$ for $X(\pi_1(M))$. As is common, the character of a representation whose image has some property will be described by that property, e.g.\ $\chi_\rho$ is called metabelian if $\rho(G)$ is metabelian. The set $X(G)$ is known to be algebraic \cite{CS, GM}. Given $g\in G$, there is a regular function $t_g\co X(G) \to \bb C$ defined by $t_g(\chi_\rho) = \tr(\rho(g))$ and $\{t_g, g\in G\}$ generates the ring $\bb C[X(G)]$ of regular functions on $X(G)$. Also recall that if $\rho,\rho'\in R(G)$ are such that their characters agree and $\rho$ is irreducible, then $\rho$ and $\rho'$ are conjugate \cite{CS}.

\subsection{Characters of a free group}
Let $F_r=\langle f_1,\ldots,f_r\rangle$ be the rank $r$ free group and $\mathfrak g\co F_r\to G$ a surjective homomorphism, with $g_i = \mathfrak g(f_i)$ for $1\le i\le r$. Then every element of $\bb C[X(G)]$ is a rational function in $\{t_{g_i} | 1\le i\le r\}$, $\{t_{g_ig_j} | 1\le i<j\le r\}$, and $\{t_{g_ig_jg_k} | 1\le i<j<k\le r\}$. In fact, there is some choice of $\mathfrak g$ so that $\bb C[X(G)]$ is generated by rational functions in $\{t_{g_i} | 1\le i\le r\}$, $\{t_{g_ig_j} | 1\le i<j\le r, i\le 3\}$, and $t_{g_1g_2g_3}$ \cite{V, M}.

Some lemmas used in \cite{M} to obtain this last result will be important to us. First we recall the ``Fricke relation'' (originally known in \cite{V}) on characters in $X(F_r)$. Let $X_1,X_2,X_3\in\SL_2\bb C$ be arbitrary. Define $x_i = \tr X_i$, $y_{ij} = \tr(X_iX_j)$, and $z_{ijk} = \tr(X_iX_jX_k)$ for $i,j,k\in\{1,2,3\}$. Note that $y_{ii} = x_i^2 - 2$ for $i=1,2,3$. Denoting the commutator of $X_i,X_j$ by $[X_i,X_j]$ as usual, note also that $\tr([X_i,X_j]) = x_i^2+x_j^2+y_{ij}^2-x_ix_jy_{ij}-2$. Define 
		\al{
		P 	&= x_1y_{23} + x_2y_{31} + x_3y_{12} - x_1x_2x_3 \\
		\intertext{and}
		Q 	&= x_1^2+x_2^2+x_3^2 + y_{12}^2 + y_{31}^2 + y_{23}^2 + y_{12}y_{31}y_{23} - x_1x_2y_{12} - x_3x_1y_{31} - x_2x_3y_{23} - 4.
		}
Then $z_{123}$ and $z_{132}$ are the roots of $z^2 + Pz + Q = 0$. The \textit{Fricke discriminant} is denoted 
		\begin{equation}
		\Delta(X_1,X_2,X_3) = P^2 - 4Q.
		\label{def-Frickediscriminant}
		\end{equation}

\begin{lem}[\cite{M}] For any choice of $X_1,X_2,X_3\in\SL_2\bb C$, $\Delta(X_1,X_2,X_3) = 0$ if and only if there exists $T\in\SL_2\bb C$ with $\text{\emph{tr}}(T) = 0$ such that $(TX_i)^2 = -\text{Id}$ for $i=1,2,3$, unless $\text{\emph{tr}}([X_i,X_j]) = 2$ for all $i,j\in\{1,2,3\}$.
\label{lem-MagnusTraceless}
\end{lem}

From the proof of Lemma \ref{lem-MagnusTraceless} in \cite{M} the determinant of $\left[\tr(X_iX_j^{-1})\right]_{1\le i,j\le 4}$, for $X_4=\text{Id}$, is zero if and only if $\D(X_1,X_2,X_3)=0$.


Another proposition from which we will derive some useful corollaries is shown in \cite{GM}.

\begin{lem}[Proposition 4.8, \cite{GM}] Let $X_i, i=1,2,3,4$ and $Y_j, j=1,2,3,4$ be in $\SL_2\bb C$. Then
		\[\det\left[\tr(X_iY_j) - \tr(X_iY_j^{-1})\right] = 0.\]
\label{lem-GMTraceMatrix}
\end{lem}

\begin{cor}For $m\ge 5$, let $X_i, i=1,\ldots,m$ be matrices in $\SL_2\bb C$. Then the $m\times m$ trace matrix $\left[\tr(X_iX_j^{-1})\right]_{1\le i,j\le m}$ has determinant zero.
\label{cor-tracematrixisnotrank5}
\end{cor}

\begin{proof}First assume that $X_1$ is the identity. Recall the basic trace identities $\tr(A) = \tr(A^{-1})$ and $\tr(AB)+\tr(AB^{-1}) = \tr(A)\tr(B)$ which hold for any $A,B\in\SL_2\bb C$. In particular we have $\tr(AB^{-1}) - \frac12\tr(A)\tr(B) = \frac12(\tr(AB^{-1}) - \tr(AB))$. Lemma \ref{lem-GMTraceMatrix} implies that every size 4 minor in $\left[\tr(X_iX_j^{-1}) - \tr(X_iX_j)\right]_{2\le i,j\le m}$ is zero. Hence, since $X_1$ is the identity,

		\[\det\left[\tr(X_iX_j^{-1})\right] = \left|\begin{array}{c | c} 2 & \tr(X_2) \qquad \tr(X_3) \qquad \ldots \qquad \tr(X_m) \\ \hline \\ {\bf 0} & \text{\Large $\frac12\left[\tr(X_iX_j^{-1}) - \tr(X_iX_j)\right]$}_{2\le i,j\le m} \end{array}\right| = 0.\]
Now we note that having $X_1$ be the identity was only cosmetic since, in the more general case, if we set $X_1'=\text{Id}$ and $X_i' = X_iX_1^{-1}$ the original matrix under consideration is precisely that corresponding to $X_1',\ldots,X_m'$ and $(X_1')^{-1},\ldots,(X_m')^{-1}$.
\end{proof}

An elementary consequence of Corollary \ref{cor-tracematrixisnotrank5} is the following.
\begin{cor}Suppose that $X_i\in\SL_2\bb C$, $i=1,\ldots,5$. Then the determinant
		\[\left|\begin{matrix}2 & \tr(X_1X_2^{-1}) & \tr(X_1X_3^{-1}) & \tr(X_1X_5^{-1}) \\ \tr(X_2X_1^{-1}) & 2 & \tr(X_2X_3^{-1}) & \tr(X_2X_5^{-1}) \\ \tr(X_3X_1^{-1}) & \tr(X_3X_2^{-1}) & 2 & \tr(X_3X_5^{-1}) \\ \tr(X_4X_1^{-1}) & \tr(X_4X_2^{-1}) & \tr(X_4X_3^{-1}) & \tr(X_4X_5^{-1})\end{matrix}\right|\]
is a multiple of $\det\left[\tr(X_iX_j^{-1})\right]_{1\le i,j\le 4}$, and so is zero if the latter is zero.
\label{cor-switchcolumn}
\end{cor}

\begin{rem}
Suppose for $X_1,X_2,X_3$ that some $T$ is found as in Lemma \ref{lem-MagnusTraceless}. While not mentioned in \cite{M}, if there are $X_j\in\SL_2\bb C$, $j>3$, such that $\D(X_1,X_2,X_j)=0$ also then $(TX_j)^2 = -\text{Id}$. To see this, note the method used to find
\[T = \begin{pmatrix}t_{11} & t_{12} \\ t_{21} & -t_{11}\end{pmatrix}\]
was to consider $X_i=\begin{pmatrix}a_i & b_i \\ c_i & d_i\end{pmatrix}\in\SL_2\bb C$ and find $(t_{11},t_{21},t_{12})$ orthogonal to all $(a_i-d_i,b_i,c_i)$, for $i=1,2,3$. The square of the determinant of this system equals $\D(X_1,X_2,X_3) $. In addition, $T$ can be shown to be invertible provided the commutator, of $X_1,X_2$ say, does not have trace 2. But, in particular, $\tr([X_1,X_2])\ne2$ means the two rows $(a_i-d_i,b_i,c_i)$ where $i=1,2$ are independent and so the same $(t_{11},t_{21},t_{12})$ is orthogonal to $(a_j-d_j,b_j,c_j)$, $j>3$.
\label{rem-Tappliesmore}
\end{rem}

\subsection{Trace-free characters of $E_K$}
\label{sec-TFchars}

Given an knot $K\subset S^3$, a \textit{meridian} in $\pi_1(E_K)$ is an element represented by a loop that is the oriented boundary of a disk $D$, embedded in $S^3$, such that there is one point in $D\cap K$, the intersection being transverse. Note that if we choose the basepoint of  $\pi_1(E_K)$ on the boundary torus this picks out a meridian, the boundary of a meridian disk of our torus neighborhood of $K$. For any meridian, conjugates of that meridian generate $\pi_1(E_K)$.

A representation $\rho\in\text{Hom}(\pi_1(E_K),\SL_2\bb C)$ is called \textit{trace-free} if $\tr(\rho(m))=0$ for a meridian $m$. The set of characters of trace-free representations will be written $X_{TF}(E_K)$, which, by the above, is not affected by choice of $m$. As $X(E_K)$ is known to be algebraic, so is $X_{TF}(E_K)$. 

In fact, Nagasato \cite{Nag} has given a description of $X_{TF}(E_K)$ as an algebraic set that will be useful to us. Given a knot diagram $D$ of $K$ with $n$ crossings, let $\langle m_1,\ldots,m_n | r_1,\ldots,r_n\rangle$ be a Wirtinger presentation of $\pi_1(E_K)$. Recall that each $m_i, i=1,\ldots,n$ is a meridian (associated to a segment labeled $i$) and that a relation $r_j$ is determined by a crossing. 

Let $(\{x_{ab}, 1\le a< b\le n\}, \{x_{abc}, 1\le a<b<c\le n\})$ be coordinates on $\bb C^{\binom n2 + \binom n3}$. The $x_{ab}$ (resp.\ $x_{abc}$) coordinate of $p\in\bb C^{\binom n2 + \binom n3}$ is denoted $x_{ab}(p)$ (resp.\ $x_{abc}(p)$). Below $(i,j,k)$ denotes a crossing of the knot diagram, where $i$ is the label of the overcrossing strand and $j,k$ the labels of the two undercrossing strands at that crossing. We let $x_{aa}=2, x_{ab} = x_{ba}$ and $x_{a_{\sg(1)}a_{\sg(2)}a_{\sg(3)}} = \text{sign}(\sg)x_{a_1a_2a_3}$. In addition, given $1\le a_1<a_2<a_3\le n$ and $1\le b_1<b_2<b_3\le n$ set ${\bf a} = (a_1,a_2,a_3)$ and ${\bf b} = (b_1,b_2,b_3)$ and for a point $p\in \bb C^{\binom n2 + \binom n3}$ write
		\[T({\bf a},{\bf b})(p) := \det\left[x_{a_ib_j}(p)\right]_{1\le i,j\le 3}.\]
The main theorem of \cite{Nag} is that $X_{TF}(E_K)$ is equivalent to the set of points in $\bb C^{\binom n2 + \binom n3}$ cut out by the following equations, where $i,j,k$ run over all crossings $(i,j,k)$ in the diagram:
	\begin{align}
		x_{ja} + x_{ka} &= x_{ji}x_{ia},\quad 1\le a\le n \label{F2}\\
		x_{jab} + x_{kab} &=  x_{ji}x_{iab}, \quad 1\le a < b\le n \label{F3}\\
		x_{a_1a_2a_3}x_{b_1b_2b_3} &= \frac12T({\bf a},{\bf b}), \quad 1\le a_1 < a_2 < a_3\le n, 1\le b_1 < b_2 < b_3\le n \label{H}\\
		0 &= \left|\begin{matrix}x_{11} & x_{12} & x_{1a} & x_{1b} \\ x_{21} & x_{22} & x_{2a} & x_{2b} \\ x_{a1} & x_{a2} & x_{aa} & x_{ab} \\ x_{b1} & x_{b2} & x_{ba} & x_{bb}\end{matrix}\right|, \quad 3\le a < b\le n \label{R}.
	\end{align}

Nagasato has the coordinates corresponding to the character of $\rho\co \pi_1(E_K)\to\SL_2\bb C$ by $x_{ab} \leftrightarrow -\tr(\rho(m_am_b))$ and $x_{abc} \leftrightarrow -\tr(\rho(m_am_bm_c))$. For convenience we put $x_{ab} \leftrightarrow \tr(\rho(m_am_b^{-1}))$, which is equivalent as meridians are trace free. Nagasato obtains these equations using the (trace-free) Kauffman Bracket Skein Algebra of $E_K$ at $t=-1$. A handlebody with $n$ handles is obtained from the knot diagram and the algebra is the quotient of the KBSA of that handlebody by the ``fundamental ideal'' given by relations coming from the crossings. Hereafter we identify $X_{TF}(E_K)\subset\bb C^{\binom n2 + \binom n3}$ with the affine variety just described, the implicit choice of some diagram of $K$ being understood.

Note that (\ref{R}) is a consequence of Lemma \ref{lem-GMTraceMatrix} since the meridians are trace-free. Moreover, one can replace indices 1 or 2 in the determinant of (\ref{R}) by another index and the determinant remains zero. Also, in Section 3.2 of \cite{Nag} it is shown that equations (\ref{F3}) are implied by (\ref{F2}), (\ref{H}), and (\ref{R}).

\begin{lem}Suppose that $q\in\bb C^{\binom n2}$ satisfies equations (\ref{F2}) and (\ref{R}). Choose $1\le a_1<a_2<a_3\le n$ such that $\delta({\bf a}) := \frac12 T({\bf a},{\bf a})(q) \ne 0$, if such exists, and choose a square root $\sqrt{\delta({\bf a})}$. If no such ${\bf a}$ exists set ${\bf a} =(1,2,3)$. There exists a unique $p\in X_{TF}(E_K)$ with $x_{ab}(p) = x_{ab}(q)$ for $1\le a<b\le n$ and $x_{a_1a_2a_3}(p) = \sqrt{\delta({\bf a})}$.
\label{lem-SqrootCorrespondence}
\end{lem}

\begin{proof}Define $x_{ab}(p) = x_{ab}(q)$ and $x_{a_1a_2a_3}(p) = \sqrt{\delta({\bf a})}$ ($=0$ if all $T({\bf a},{\bf a})$ are zero). For $1\le b_1<b_2<b_3\le n$, use (\ref{H}) and $x_{a_1a_2a_3}(p) = \sqrt{\delta({\bf a})}$ to define $x_{b_1b_2b_3}(p)$. We need only show, given $1\le a_1'<a_2'<a_3'\le n$ with some $a_i' \ne a_i$, that $x_{a_1'a_2'a_3'}x_{b_1b_2b_3}$ satisfies (\ref{H}) for all $1\le b_1<b_2<b_3\le n$. This is equivalent to demonstrating
		\begin{equation}
		T({\bf a},{\bf b})T({\bf a},{\bf a'}) = T({\bf b},{\bf a'})T({\bf a},{\bf a}).
		\label{eqn-HexagonConsistency}
		\end{equation}
Consider the case that $a_1 = a_1', a_2 = a_2',$ but $a_3 \ne a_3'$. From (\ref{R}) we have that for any $1\le b\le n$,
			\[\left|\begin{matrix}x_{b_1a_1} & x_{b_1a_2} & x_{b_1a_3'} & x_{b_1a_3} \\ x_{b_2a_1} & x_{b_2a_2} & x_{b_2a_3'} & x_{b_2a_3} \\ x_{b_3a_1} & x_{b_3a_2} & x_{b_3a_3'} & x_{b_3a_3} \\ x_{ba_1} & x_{ba_2} & x_{ba_3'} & x_{ba_3}\end{matrix}\right| = 0\]
 at $q$. The same elementary argument which shows that Corollary \ref{cor-switchcolumn} follows from Corollary \ref{cor-tracematrixisnotrank5} may be applied here to see $T({\bf b},{\bf a'}) = \alpha T({\bf b},{\bf a})$ for some $\alpha$. Moreover, considering the linear relations among the columns when setting $b=a_1,a_2$, or $a_3$ in the matrix above, we see $T({\bf a},{\bf a'}) = \alpha T({\bf a},{\bf a})$ for the same $\alpha$. As $T({\bf a},{\bf b}) = T({\bf b},{\bf a})$, this shows that (\ref{eqn-HexagonConsistency}) holds in this case. 

 The argument did not depend on which component of ${\bf a}$ differed from that of ${\bf a'}$, and only required that $x_{a_1a_2a_3}x_{b_1b_2b_3}$ and $x_{a_1a_2a_3}x_{a_1'a_2'a_3'}$ satisfy (\ref{H}). Hence, successively replacing ${\bf a}$ for a triple on which the equality has already been shown, the general case follows.
\end{proof}

There is an involution $\iota\co R(E_K)\to R(E_K)$ studied in \cite{NY}. It is defined by $\iota(\rho)(g) = (-1)^{\text{lk}(g)}\rho(g)$. Here $\text{lk}(g)$ is the linking number of $g$ with $K$, i.e. $\text{lk}\co \pi_1(E_K)\to\ints$ is the homomorphism determined by sending a meridian to $1$. This gives an involution, $\chi_\rho\leftrightarrow\chi_{\iota(\rho)}$ on $X(E_K)$ as well (which we denote by $\iota$ also), which restricts to an involution on $X_{TF}(E_K)$. 

\begin{prop}[\cite{NY}] The set of fixed points $X(E_K)^{\iota}$ is contained in $X_{TF}(E_K)$ and $X(E_K)^{\iota}$ equals the union of the irreducible metabelian characters and the character of the abelian representation determined by 
		\[m\mapsto \begin{pmatrix}\sqrt{-1}&0 \\ 0&-\sqrt{-1}\end{pmatrix}.\]
\label{prop-NYfixedpts}
\end{prop}

The set of irreducible metabelian characters in $X(E_K)$ equals the set of binary-dihedral $\SU(2)$ characters (see e.g.\ \cite[Proposition 2]{NY}). They are finite in number, of cardinality $(|\D_K(-1)|-1)/2$, where $\D_K(t)$ is the Alexander polynomial \cite{klassen}.

From the identification of $X_{TF}(E_K)$ with the algebraic set given by (\ref{F2}), (\ref{H}), and (\ref{R}) note that a point in $X_{TF}(E_K)$ is fixed by $\iota$ precisely if it projects on the last $\binom n3$ factors as zero. That is, if and only if $x_{abc} = 0$ for all $1\le a<b<c \le n$. By the equations (\ref{H}), this implies that $q\in X_{TF}(E_K)$ is a fixed point of $\iota$ if and only if the matrix $\left[x_{ab}(q)\right]_{1\le a,b\le n}$ has rank less than three.

\section{The cord algebra}
\label{SecBGC_K}

\subsection{Definitions}
\label{sec-HCdefinitions}

Given a knot $K\subset S^3$ with a basepoint $\ast$ on $K$, define a \textit{cord} of $(K,\ast)$ to be a path $\gamma\co [0,1]\to S^3$ such that $\gamma^{-1}(K)=\set{0,1}$ and $\ast\not\in\gamma([0,1])$.

The algebra $HC_0(K)$ in grading zero of knot contact homology admits a description where it is generated by homotopy classes of cords over a Laurent polynomial ring $\bb Z[\mu^{\pm1},\lambda^{\pm1}]$, that polynomial ring being identified with the group ring of the first homology of a torus over $K$ in the unit cotangent bundle \cite{Ng08,Ng14}. In our discussion on augmentations below, we will refer to this algebra. However, for the purpose of the results in this paper, it suffices to give a presentation of the commutative quotient ring $\hc(K)$ from \cite{Ng05b}. An analogue of our results here on representations into higher rank Lie groups, if such exists, will likely require consideration of a different quotient of $HC_0(K)$.

\begin{defn} Consider the commutative ring (with 1) generated by unoriented homotopy classes of cords of $K$, forgetting the basepoint. That is, homotopy of cords allows endpoints to move along $K$, even across the basepoint; and by an unoriented homotopy class we mean that the equivalence relation puts cords $\gamma$ and $\bar\gamma$ in the same class, where $\bar\gamma\co[0,1]\to S^3$ is defined by $\bar\gamma(t) = \gamma(1-t)$. The product is a formal product on these classes. The \emph{abelian cord ring} $\hc(K)$ is the quotient of this ring by the ideal generated by the relations

\begin{figure}[ht]
\begin{tikzpicture}[scale=0.5,>=stealth]
	\draw(-3,0) node {(1)}; 
	\draw[very thick,->](1,1)--(-1,-1);
	\draw[gray](0,0) ..controls (-1.5,0)and(0,1.5).. (0,0);
	\draw(2,-.25) node {$=$};
	\draw(3.5,-.25) node {$2$;};
	
	\draw(-3,-3) node {(2)}; 
	\draw[very thick,->](1,-2)--(-1,-4);
	\draw[->, white,line width=3pt](-1,-2)--(1,-4);
	\draw[gray](-1,-2)--(1,-4);	
	\draw(2,-.25-3) node {$+$};
	\draw[gray](3.5,-2)--(5.5,-4);
	\draw[->,very thick,draw=white,double=black,double distance=1pt](5.5,-2)--(3.5,-4);
	\draw(6.5,-.25-3)node{$=$};
	\draw[very thick,->](9,-2)--(7,-4);
	\draw[gray](7,-2)--(8,-3);
	\filldraw(9.5,-3) circle (1pt);
	\draw[very thick,->](12,-2)--(10,-4);
	\draw[gray](11,-3)--(12,-4);
	
	\end{tikzpicture}
\caption{Relations in $\hc(K)$}
\label{fig-cordrelations}
\end{figure}
\label{defnAbCordRing}
\end{defn}

In the above definition, relations are between any cords that differ only locally, as shown. The knot $K$ is depicted more thickly. Also the figures depict a relation in 3-dimensions. 

The map sending the homotopy class $[\gamma]$ of a cord $\gamma$ to $-[\gamma]$ induces an isomorphism with the algebra defined in \cite{Ng05b}.

\comment{
Recalling the homotopy formulation $\cl P_K$ of the cord algebra $HC_0(K)$, we have the following. 

\begin{prop} Let $P_K$ denote the underlying set of the knot group $\pi_K$, where we write $[\gamma]\in P_K$ for $\gamma\in\pi_K$. In $\pi_K$ let $e$ be the identity, $m$ a choice of meridian, and $\ell$ the preferred longitude of $K$. The ring $\hc(K)$ is isomorphic to the commutative ring generated by $P_K$, modulo the relations:
	\en{
	\item $[e]=2$;
	\item $[m\gamma]=-[\gamma]=[\gamma m]$ and $[\ell\gamma]=[\gamma]=[\gamma\ell]$, for any $\gamma\in\pi_K$;
	\item $[\gamma_1\gamma_2]-[\gamma_1m\gamma_2]=[\gamma_1][\gamma_2]$ for any $\gamma_1,\gamma_2\in\pi_K$.
	}
\label{PropPi1Cords}
\end{prop}

There is an isomorphism $F_{\cl P}\co \cl P_K\to HC_0(K)$ defined as follows. Suppose the basepoint $x$ for the group $\pi_K$ is on the boundary torus. Choose a fixed path $p$ from a nearby point on $K$ to $x$. Let $\overline p$ denote $p$ with reversed orientation. If $g\in\pi_K$ is represented by a loop $\gamma$, define $F_{\cl P}([g])$ to be $\mu^{\text{lk}(\gamma,K)}$ times the cord given by the concatenation $p\gamma\overline p$.

The commutative ring described in Proposition \ref{PropPi1Cords} is the quotient of $\cl P_K$ by the kernel of the composition of the quotient map $HC_0(K)\to\hc(K)$ with $F_{\cl P}$.

} 

\subsection{The cord algebra from a knot diagram}
\label{sec-HCknotdiagram}

The abelian cord ring can be described as a polynomial ring from a knot diagram of $K$. This description will be useful in relation to $\bb C[X_{TF}(E_K)]$. Suppose the diagram has $n$ components (the segments along the knot between two consecutive under-crossings). Label these components by $\{1,\ldots,n\}$ and let $c_{ij}$ denote a homotopy class of cord that begins on (the preimage under projection of) component $i$, ends on $j$, and passes above $K$ in its interior. Then $\hc(K)$ is generated by the set $\{c_{ij} | 1\le i<j\le n\}$.

\begin{thm}[Ng05b] Write $\cl A^{ab}_n =\ints\langle c_{ij} | 1\le i < j\le n\rangle$. Denote a crossing in an $n$-crossing knot diagram $D$ of $K$ by $(i,j,k)$ where $i$ is the label on the overcrossing strand and $j,k$ are the two undercrossing strands. Define $\cl I_D$ to be the ideal in $\cl A^{ab}_n$ generated by 
	\begin{equation}
	\{c_{lj} + c_{lk} - c_{li}c_{ij}\ |\ (i,j,k)\text{ is a crossing, }1\le l\le n\},
	\label{eqn-diagram-defneqns}
	\end{equation}
where we set $c_{ll} = 2$ and $c_{ml} = c_{lm}$ for $l<m$. Then $\hc(K) \cong \cl A^{ab}_n / \cl I_D$.
\label{thm-knotdiagram}
\end{thm}

Unless otherwise stated, we assume a diagram $D$ of $K$ has been chosen and will identify $\hc(K)$ with $\cl A^{ab}_n / \cl I_D$. Note that $c_{ij}=c_{ik}$ in $\hc(K)$ if $(i,j,k)$ is a crossing (the realizations as cords are also clearly homotopic). Below we will need the following notation. Given an $n$-crossing diagram $D$, let $C_D = \left[ c_{ij}\right]$ be the $n\times n$ matrix over $\hc(K)$, identifying $c_{ji}=c_{ij},\ 1\le i<j\le n$, and $c_{ii}=2$.

\subsection{Augmentations}
\label{sec-HCaugmentations}

An \textit{augmentation} of $K$ is a ring homomorphism $\ve\in\text{Hom}(HC_0(K),\bb C)$. In a more general setting augmentations are defined in the category of differential graded algebras, however $\text{Hom}(HC_0(K),\bb C)$ will suffice for our purposes.

\begin{defn}We call an augmentation $\ve\in\text{Hom}(HC_0(K),\bb C)$ \textit{reflective} if it factors through $\hc(K)$.
\label{defn-abaug}
\end{defn}

\begin{defn}Given a knot diagram $D$ and $\ve$ an augmentation, let $C_D(\ve)$ be the complex matrix obtained by applying $\ve$ to each entry of $C_D$, defined in the previous subsection. We define the \textit{rank} of $\ve$ to be the rank of $C_D(\ve)$.
\label{defn-augrank}
\end{defn}

As in the introduction, we write $\mathfrak r^{ab}(K)$ for the set of ranks of reflective augmentations for $K$. See the remarks following Theorem \ref{thm-correspondenceAugRep}.

\begin{example}
This example will be useful to us in Section \ref{sec-unitary}. Let $K_0$ be the knot represented by the labeled diagram on the left of Figure \ref{fig-firstexample}, which is the knot $5_2$ in knot tables. Homotopy on cords lets us see that $c_{12}=c_{13}=c_{35}=c_{45}$ and that $c_{14}=c_{24}=c_{25}$ in $\hc(K_0)$.
\begin{figure}[ht]
\begin{center}
\begin{tikzpicture}

\node[anchor=south west,inner sep=0] at (0,1) {\includegraphics[scale=0.5]{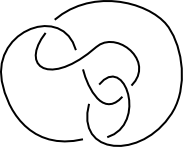}};

\draw(1.5,0.5) node {$5_2$};

\draw(2.3,2.5) node {$1$};
\draw(0.3,1.6) node {$3$};
\draw(1.1,1.9) node {$4$};
\draw(1.6,2.7) node {$5$};
\draw(2,1.5) node {$2$};

\end{tikzpicture}
\end{center}
\caption{An example augmentation.}
\label{fig-firstexample}
\end{figure}
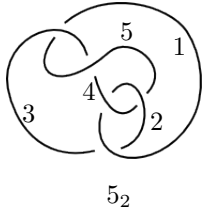

By applying equations (\ref{eqn-diagram-defneqns}) we obtain $c_{15}+2 = c_{13}^2$, $c_{34}+2 = c_{35}^2$, and $c_{23}+2 = c_{12}^2$, which by the previous equations imply that $c_{15}=c_{34}=c_{23}$. Also we have that $c_{13}+c_{14} = c_{15}c_{35}$ and $c_{23}+c_{24} = c_{25}c_{35}$. These equations imply that $c_{15} = c_{12}^2-2$, $c_{14} = c_{12}^3-3c_{12}$, and  $c_{15}=(c_{12}-1)c_{14}$. It turns out that these equations suffice and $\hc(K)$ is generated by one cord, say $c_{12}$. Using the relation $c_{12}^2-2=(c_{12}-1)c_{14}$ we obtain that $\hc(K)\cong\bb Z[c_{12}]\big/((c_{12}-2)(c_{12}^3+c_{12}^2-2c_{12}-1))$. Sending $c_{12}$ to any root of $x^3+x^2-2x-1$ determines a reflective augmentation of rank 2.

\label{example-1}
\end{example}

Choose any meridian $m$ in $\pi_1(E_K)$. Define $R_{\aug}(E_K,k)$ to be the subset of $\text{Hom}(\pi_1(E_K),\GL_k\bb C)$ consisting of those homomorphisms $\rho$ such that $\rho(m)$ is diagonalizable, and has 1 as an eigenvalue of multiplicity $k-1$. Let $R_{\aug}(E_K) = \cup_{k\ge1}R_{\aug}(E_K,k)$. The set $R_{\aug}(E_K)$ is closely related to augmentations of $K$. 

Let us note a few properties of $R_{\aug}(E_K)$. First, given $K$ there is an $n_K$ such that all irreducible elements in $R_{\aug}(E_K)$ are contained in the union $\cup_{k=1}^{n_K}R_{\aug}(E_K,k)$ and $n_K$ is at most the number of meridians needed to generate $\pi_1(E_K)$ \cite{Cor14}. Also, as with any $\GL_k\bb C$ representation of a group generated by conjugates of one element, a choice of $k^{th}$ root of $\det(\rho(m))$ makes each $\rho\in R_{\aug}(E_K,k)$ a 1-dimensional representation times an element of $\text{Hom}(\pi_1(E_K),\SL_k\bb C)$.

Given $\rho\in R_{\aug}(E_K)$ we can construct an augmentation $\ve_\rho\co HC_0(K)\to\bb C$ \cite[Section 5]{Ng14}. For augmentations with $\ve(\mu)\ne 1$ this correspondence is 1-1 on equivalence classes of irreducible elements of $R_{\aug}(E_K)$.

\begin{thm}[\cite{Cor14b}]If $\ve\co HC_0(K)\to\bb C$ is an augmentation and $\ve(\mu)\ne1$ there is some $k>0$ and an irreducible $\rho\in R_{\aug}(E_K,k)$ with $\ve_\rho = \ve$. Any irreducible $\rho'\in R_{\aug}(E_K)$ with $\ve_{\rho'} = \ve$ is equivalent to $\rho$. Moreover, if $D$ is a knot diagram of $K$ then $k = \text{rank}(\ve_\rho) = \text{rank}(C_D(\ve_\rho))$ for any irreducible $\rho\in R_{\aug}(E_K,k)$.
\label{thm-correspondenceAugRep}
\end{thm}

As a relevant aside, for any non-trivial knot $K$ work of Kronheimer and Mrowka shows there is an irreducible representation $\sg\co \pi_1(E_K) \to SU(2)$ with $\sg(m)$ a diagonal matrix with eigenvalues $\pm\sqrt{-1}$ \cite[Corollary 7.17]{KM}. Hence, upon multiplying by the 1-dimensional map determined by $m\mapsto \sqrt{-1}$, every non-trivial knot $K$ has an irreducible representation in $R_{\aug}(E_K,2)$ with $\det(\rho(m)) = -1$, so inducing a rank 2 augmentation. However, it may be that $2\not\in\mathfrak r^{ab}(K)$, i.e.\ there are knots with no rank 2 \textit{reflective} augmentation (see Corollary \ref{cor-Det1}).

Recall that if a matrix with rank $r$ is diagonalizable, then by considering the characteristic polynomial we get a principal submatrix of size $r$ which has non-zero determinant. This shows that reflective augmentations have the following useful property.
\begin{lem}Suppose that the rank of $\ve\co \hc(K)\to\bb C$ is $r$. Then for any knot diagram $D$ of $K$ there is a principal submatrix of $C_D(\ve)$ of size $r$ with non-zero determinant.
\label{lem-principalminors}
\end{lem}
\begin{proof}As $\ve$ is reflective it satisfies $\ve(c_{ij}) = \ve(c_{ji})$ and $\ve(\mu) = -1$. Thus $C_D(\ve)$ is symmetric and therefore can be diagonalized.
\end{proof}


\begin{lem}Given an $n$-crossing diagram $D$, if $\ve$ is a reflective augmentation such that $\ve(c_{rs}) = \pm2$ for all $1\le r<s\le n$, then $\ve(c_{rs}) = 2$ for all $1\le r<s\le n$, and so $\text{rank}(\ve) = 1$.
\label{lem-rank1aug}
\end{lem}
\begin{proof}Let $i$ be a label on a component of $D$ and suppose $\ve_{ij} = -2$ for some $1\le j\le n$. Choose a crossing at which $j$ is undercrossing, and let $j'$ be the other undercrossing strand. As $\ve(c_{rs})\pm2$ for all $r<s$, by the equations (\ref{eqn-diagram-defneqns}) it must be that $\ve(c_{ij'}) = -2$ also. Moving along the knot in $D$ we have $\ve(c_{is}) = -2$ for all $1\le s\le n$. In particular, $2 = \ve(c_{ii}) = -2$, a contradiction.
\end{proof}

\begin{prop}Suppose $\ve$ is a reflective augmentation of rank at least two, where $D$ has $n$ crossings. There is a labeling of $D$ so that for $2<k\le n$ we may find matrices $A_2,A_k\in\SL_2\bb C$ so that, upon setting $A_1$ to the identity, we get $\tr(A_rA_s^{-1}) = \ve(c_{rs})$ for $r,s \in \{1,2,k\}$. The trace of the commutator $[A_2,A_k]$ is 2 if and only if the principal submatrix in $C_D(\ve)$ corresponding to rows $1,2,k$ has zero determinant. Moreover, for $1\le l\le n$ there is $A_l\in\SL_2\bb C$ so that $\tr(A_rA_s^{-1}) = \ve(c_{rs})$ for $r,s \in\{1,2,k,l\}$ if either $\tr([A_2,A_k])\ne 2$ or the principal submatrix of $C_D(\ve)$ corresponding to rows $1,2,l$ has zero determinant. 
\label{prop-augtotrace}
\end{prop}
\begin{proof}
Let us abbreviate $\ve_{rs}:= \ve(c_{rs})$ for $1\le r<s \le n$. Since $\ve$ does not have rank one there is some $i<j$ such that $\ve_{ij} \ne \pm2$ (Lemma \ref{lem-rank1aug}) and we may assume $i=1, j=2$ by a relabeling of $D$. Setting $d$ to a root of $d^2 - \ve_{12}d + 1 = 0$ and $a = \frac{\ve_{2k} - \ve_{1k}d}{\ve_{12} - 2d}$ (note that $\ve_{12} \ne 2d$ as $\ve_{12}\ne\pm 2$), we take
			\[A_2 = \begin{pmatrix}d & 0 \\ 0 & \ve_{12}-d\end{pmatrix} \quad\text{and}\quad A_k = \begin{pmatrix}a & 1 \\ a(\ve_{1k}-a) - 1 & \ve_{1k} - a\end{pmatrix}.\]
By computation, $\tr(A_kA_2^{-1}) = \ve_{2k}$. 

Note that $\tr([A_2,A_k]) = \tr(A_2A_kA_2^{-1}A_k^{-1}) = \tr(A_kA_2A_k^{-1}A_2^{-1}) = \tr([A_2,A_k^{-1}])$ in $\SL_2\bb C$. Hence
		\[\tr([A_2,A_k]) = \tr(A_2)^2 + \tr(A_k^{-1})^2 + \tr(A_2A_k^{-1})^2 - \tr(A_2)\tr(A_k^{-1})\tr(A_2A_k^{-1}) - 2 = \ve_{12}^2 + \ve_{1k}^2 + \ve_{2k}^2 - \ve_{12}\ve_{1k}\ve_{2k} - 2.\]
On the otherhand,
				\[\left|\begin{matrix}2 & \ve_{12} & \ve_{1k} \\ \ve_{12} & 2 & \ve_{2k} \\ \ve_{1k} & \ve_{2k} & 2\end{matrix}\right| = -2(\ve_{12}^2+\ve_{1k}^2+\ve_{2k}^2 - \ve_{12}\ve_{1k}\ve_{2k} - 4),\]
and so this principal submatrix of $C_D(\ve)$ is singular if and only if $\tr([A_2,A_k]) = 2$.

Choose $1\le l\le n$. First suppose that $\tr([A_2,A_k]) \ne 2$. Define $a_l = \frac{\ve_{2l} - \ve_{1l}d}{\ve_{12} - 2d}$. For a choice of $b_l,c_l$ that will soon be made, we set
			\begin{equation}A_l = \begin{pmatrix}a_l & b_l \\ c_l & \ve_{1l} - a_l\end{pmatrix}.
			\label{eqn-forAl}
			\end{equation}
Note that
	\al{
	\tr(A_2\begin{pmatrix}\ve_{1l} - a_l & -b_l \\ -c_l & a_l\end{pmatrix}) 	&= d\ve_{1l} - d\frac{\ve_{2l} - \ve_{1l}d}{\ve_{12} - 2d} + (\ve_{12} - d)\frac{\ve_{2l} - \ve_{1l}d}{\ve_{12} - 2d} \\
																		&= \ve_{2l}.
	}
Also
	\al{
	\tr(A_k\begin{pmatrix}\ve_{1l} - a_l & -b_l \\ -c_l & a_l\end{pmatrix}) 	&= a(\ve_{1l} - a_l) - c_l + b_l - b_la(\ve_{1k}-a) + (\ve_{1k} - a)a_l
	}
so we set $c_l = a(\ve_{1l} - a_l) - \ve_{kl} + b_l - b_la(\ve_{1k}-a) + (\ve_{1k} - a)a_l$ to obtain $\tr(A_kA_l^{-1}) = \ve_{kl}$. Furthermore we can take $A_l$ to be in $\SL_2\bb C$ as the equation $\det(A_l)=1$ is now quadratic in $b_l$ and $a(\ve_{1k}-a) - 1$ is the coefficient of $b_l^2$. If $a(\ve_{1k}-a) - 1$ were zero, then 
	\[\ve_{2k}\ve_{1k}\ve_{12}-\ve_{2k}^2-\ve_{1k}^2(\ve_{12}d-d^2)	= (\ve_{2k} - \ve_{1k}d)(\ve_{1k}(\ve_{12}-2d) - (\ve_{2k} - \ve_{1k}d)) = (\ve_{12}-2d)^2 = \ve_{12}^2 - 4.\]
This equality does not hold since $\tr([A_2,A_k]) \ne 2$.

On the other hand, if $\tr([A_2,A_k]) = 2$ then suppose the principal submatrix corresponding to rows $1,2,$ and $l$ has zero determinant. That is
		\begin{equation}\ve_{12}^2 + \ve_{1l}^2 + \ve_{2l}^2 - \ve_{12}\ve_{1l}\ve_{2l} - 4 = 0.
		\label{eqn-anotherminor}
		\end{equation}
Again define $A_l$ by (\ref{eqn-forAl}), choosing $c_l$ as before. Now, since $\tr([A_2,A_k]) = 2$ we have seen that the coefficient of $b_l^2$ in $\det(A_l) = 1$ is zero. It suffices, though, to note that upon setting $b_l=0$ we have 
	\[\det(A_l) = a_l(\ve_{1l} - a_l) = \frac{\ve_{12}\ve_{1l}\ve_{2l} - \ve_{1l}^2 - \ve_{2l}^2}{\ve_{12}^2 - 4} = 1,\]
where the last equality holds by (\ref{eqn-anotherminor}).
\end{proof}

\begin{rem}We note it is the case that, taking $\ve_{ii}=2$ for every $1\le i\le n$, the matrix $A_2$ in the above proposition is also in the form $\begin{pmatrix}a_2 & b_2 \\ c_2 & \ve_{12}-a_2\end{pmatrix}$, with $b_2=0$ and $a_2,c_2$ defined as for the other cases, i.e.\ $a_2= \frac{\ve_{22}-\ve_{12}d}{\ve_{12}-2d}$ and $c_2 = a(\ve_{12} - a_2) - \ve_{2k} + b_2 - b_2a(\ve_{1k}-a) + (\ve_{1k} - a)a_2$. The same is true of $A_k$, but with $b_k = 1$.
\label{rem-formofgenerators}
\end{rem}

\subsection{The variety of reflective augmentations}
\label{sec-HCvariety}

Consider a reflective augmentation $\ve$ of $K$. Then $\ve$ is determined by the point $(\ve(c_{12}),\ve(c_{13}),\ldots,\ve(c_{1n}),\ve(c_{23}),\ldots,\ve(c_{n-1,n}))$ in $\bb C^{\binom n2}$. Given an $n$-crossing diagram $D$ for $K$, define the variety of augmentations $A(D)\in \bb C^{\binom n2}$ to be the zero set of the polynomials $\{x_{lj} + x_{lk} - x_{li}x_{ij} | (i,j,k)\text{ is a crossing and }1\le l\le n\}$, where in these polynomials (as in the ideal $\cl I_D$ from Section \ref{sec-HCknotdiagram}) if $r<s$ then we take $x_{sr}$ to be $x_{rs}$ and $x_{rr}=2$. Then $\hc(K)\otimes\bb C / \sqrt{0} \cong \bb C[A(D)]$ by the Nullstellensatz. We denote coordinates of $\ve\in A(D)$ by $\ve_{rs}:=\ve(c_{rs}),\ 1\le r<s\le n$.

	\begin{figure}[ht]
        \begin{tikzpicture}[scale=0.4,>=stealth]
        \draw[->] (0,2) -- node[at start,left]{$i$} (3,0);
        \draw[draw=white,very thick,double=black,->] (0,0) -- node[at start,left]{$i+1$} (3,2);
        \draw (1.5,-0.5) node[below] {$\sigma_i^{-1}$};

        \draw[->] (6,0) -- (9,2);
        \draw[draw=white,very thick,double=black,->] (6,2) -- (9,0);
        \draw (7.5,-1) node[below] {$\sigma_i$};
        \end{tikzpicture}
        \caption{Generators of $B_N$}
        \label{fig-braidgenerators}
	\end{figure}
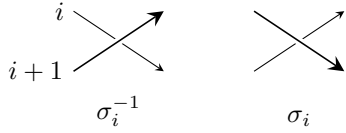

While the description of $\hc(K)$ given in Section \ref{sec-HCknotdiagram} will be important to relate augmentations to characters in $X_{TF}(E_K)$, we also desire to consider $K$ as a closed braid in order to consider an open book decomposition of the branched double-cover $\dc K$. If $D$ and $D'$ are any two knot diagrams of $K$, with crossing numbers $n$ and $n'$ respectively, then there is an isomorphism $\cl A^{ab}_n\to\cl A^{ab}_{n'}$ carrying $\cl I_D$ to $\cl I_{D'}$, and so $A(D')$ and $A(D)$ are algebraically equivalent. Thus to integrate the two perspectives we simply determine a knot diagram $D(w(\beta))$ from $\beta\in B_N$ and a word $w(\beta)$ in $\{\sg_1,\ldots,\sg_{N-1}\}$ that represents $\beta$.

Our convention is to identify $S^3$ with $\rls^3\cup\{\infty\}$ and let $z$-axis $\cup\{\infty\}$ be the braid axis. Closed braids will have an orientation so that their projection to the $xy$-plane (viewed from $z>0$) is counter-clockwise. We also fix a disk $D_N$ to agree with (the compactification of) the half-plane $\{x=0\}\cap\{y\ge 0\}$. Then $D_N$ has boundary equal to the braid axis and intersects $\hat\beta$ transversely in $N$ points. See Figure \ref{fig-closurediagram}. We assume the intersection points $\hat\beta \cap D_N = \{p_1,\ldots,p_N\}$ are in $\{z=0\}$ and $p_i$ is closer to the origin than $p_j$ if $i<j$.

For the diagram $D(w(\beta))$ draw $N$ horizontal segments, stacked vertically, and label the $i^{th}$ segment from the top by $i$ (so the top segment is 1 and the bottom $N$). Reading $w(\beta)$ from left to right, and moving along the segments from left to right, insert a crossing (with crossing information as in Figure \ref{fig-braidgenerators}) between segments $i$ and $i+1$ when $\sg_i^{\pm1}$ is encountered in $w(\beta)$. Reaching the end of $w(\beta)$, draw $N$ non-intersecting (and not self-intersecting) curves connecting the right endpoints to the left endpoints so that endpoints at the same height are connected and the curves pass above the rest of the diagram in the plane. The result is $D(w(\beta))$. For the purposes of Proposition \ref{prop-generatebyfewer} and the discussion in Section \ref{sec-dcK} we will label the components of $D(w(\beta))$ that intersect the projection of $D_N$ by $1,2,\ldots,N$ in order of increasing $y$-coordinate. An example is shown in Figure \ref{fig-closurediagram}.

Define $\gamma_{ij}$, $1\le i<j\le N$, to be an arc in the half-disk $D_N\cap\{z\ge0\}$ with endpoints on $p_i$ and $p_j$. Then $\gamma_{ij}$ determines a well-defined element of $\hc(K)$.

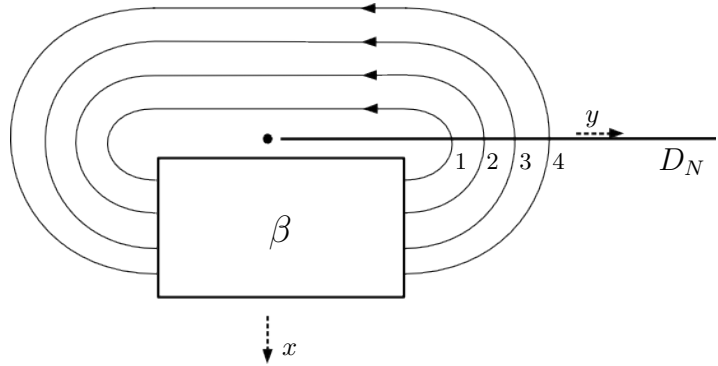
\begin{figure}[ht]
\begin{center}
\begin{tikzpicture}

\node[anchor=south west,inner sep=0] at (2,0) {\includegraphics[scale=0.6]{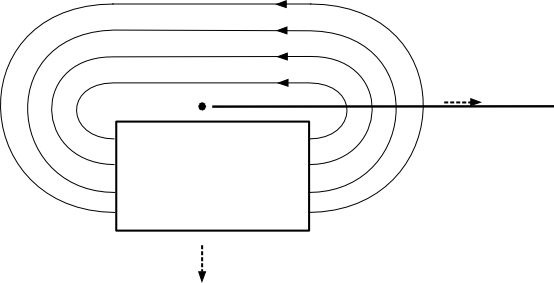}};


\draw (5.5,0) node[above right] {$x$};
\draw (9.75,3.5) node[below] {$y$};
\draw (10.5,2.7) node[right] {{\Large $D_N$}};
\draw (5.6,1.75) node {{\LARGE $\beta$}};
\draw (8,2.7) node {{\small $1$}};
\draw (8.42,2.7) node {{\small$2$}};
\draw (8.86,2.7) node {{\small $3$}};
\draw (9.3,2.7) node {{\small $4$}};

\end{tikzpicture}

\end{center}
\caption{The knot diagram $D(w(\beta))$ of $K = \hat\beta$, with the projection of $D_N$ to the $xy$-plane.}
\label{fig-closurediagram}
\end{figure}

\begin{prop}Let $K$ be the braid closure of $\beta\in B_N$ and set $D = D(w(\beta))$. Then $\hc(K) = \cl A^{ab}_n/\cl I_D$ is generated by $\{c_{ij} = \gamma_{ij}\ |\ 1\le i<j\le N\}$. For any $\ve\in A(D)$, there is a maximal non-degenerate submatrix of $C_D(\ve)$ in the principal submatrix corresponding to rows $1,2,\ldots,N$.
\label{prop-generatebyfewer}
\end{prop}
\begin{proof}
That $\hc(K)$ is generated by $\{c_{ij} = \gamma_{ij}\ |\ 1\le i<j\le N\}$ can be proved as follows. Use homotopy to move a cord in the direction of the braid until it lies entirely in $D_N$. Now one may use the relations in Figure \ref{fig-cordrelations} to express the cord as a polynomial in three other cords, also in $D_N$, which have lower complexity (that being the height function discussed in \cite[Section 2]{Ng05b}). Those cords in $D_N$ with smallest height are homotopic to the $\gamma_{ij}$, $1\le i<j\le N$.

The association made in \cite{Cor14b} of a representation in $R_{\aug}(K)$ to an augmentation was such that one can choose any subset of meridians that generate $\pi_1(E_K)$, say that these correspond to components labeled $k_1,\ldots,k_r$, and there will be an identification of the vector space of the representation with the subspace of $\bb C^r$ generated by columns of $\left[\ve_{k_ik_j}\right]_{1\le i,j\le r}$. As the meridians for components $1,\ldots,N$ in $D(w(\beta))$ generate, this implies the existence of the desired submatrix of $C_D(\ve)$.
\end{proof} 

\section{The branched double-cover $\dc K$}
\label{sec-dcK}

Let $\dc K$ be the branched double-cover over $K\subset S^3$ and $p\co \dc K\to S^3$ a degree two branched covering map. We set $\wt K = p^{-1}(K)$ and $\wt E_K = p^{-1}(E_K)$. Recall that $\pi_1(\dc K) \cong \wt\pi/\langle m^2\rangle$, where $\wt\pi = (p|_{\wt E_K})_*(\pi_1(\wt E_K))$ agrees with the kernel of the mod 2 reduction of $\text{lk}\co\pi_1(E_K)\to\bb Z$ and for a meridian $m$, $\langle m^2\rangle$ is the smallest normal subgroup of $\wt\pi$ containing $m^2$. We recall that determinant of $K$ is $|\D_K(-1)|$, which equals the order of $H_1(\dc K)$.

In what follows we identify $\pi_1(\dc K)$ with this quotient group. There is a natural involution $\tau\co \wt E_K \to \wt E_K$. Note that $\tau_*\co \wt\pi\to\wt\pi$ satisfies $\tau_*(g) = mgm^{-1}$ for all $g\in\wt\pi$. The inclusion $\wt E_K \subset \dc K$ induces a map on $\pi_1(\dc K)$ that commutes with $\tau_*$.



\subsection{Characters from trace-free characters in $X(E_K)$}
\label{sec-knotgroup-to-dcK}

Recalling terminology from Section \ref{sec-TFchars}, consider a trace-free representation $\rho\co \pi_1(E_K)\to\SL_2\bb C$. Note that restriction of $\rho$ to $\wt\pi$ determines $\rho_0\in\text{Hom}(\pi_1(\dc K),\text{PSL}_2\bb C)$; that is, $\rho(m^2) = -\text{Id}$. As discussed in \cite{NY}, given $\rho$ we may define $\wt\rho_0\in R(\dc K)$ by setting $\wt\rho_0(g) = (\sqrt{-1})^{\text{lk}(g)}\rho_0(g)$ (note $\text{lk}(g)\in 2\ints$ for $g\in\wt\pi$). Moreover, there is an induced map on characters $\Y\co X_{TF}(E_K)\to X(\dc K)$, $\Y(\chi_\rho) = \chi_{\wt\rho_0}$. 

Consider the set $R^\tau(\dc K)$ of $\tau$-equivariant representations, i.e.\ those $\rho\in R(\dc K)$ such that there is $T\in\SL_2\bb C$ so that $\rho([mgm^{-1}]) = T\rho([g])T^{-1}$ for $[g]\in\pi_1(\dc K)$. If $\rho$ is non-trivial, any such $T$ is either $\pm\text{Id}$ or $\tr(T) = 0$ \cite[Lemma 19]{NY}. Let $X^\tau(\dc K)$ be the characters of elements in $R^\tau(\dc K)$ and recall the involution $\iota\co X_{TF}(E_K)\to X_{TF}(E_K)$.

\begin{thm}[\cite{NY}] $\Y(X_{TF}(E_K)) = X^\tau(\dc K)$. Moreover, each fiber of $\Y$ is an orbit of $\iota$.
\label{thm-NYInvolution}
\end{thm}

That a $\tau$-equivariant representation $\rho$ extends to $\pi_1(E_K) = \wt\pi\ \sqcup\ m \wt\pi$ also appears in \cite[proof of Proposition 3.3.1]{Mattman}. In fact an extended representation $\wt\rho$ is explicitly described from $\rho$ and the matrix $T$, appearing above, as follows: for $g\in\wt\pi$ set $\wt\rho(g) = \rho([g])$ and for $g\in m \wt\pi$ set $\wt\rho(g) = T\rho([m^{-1}g])$.

If $T = \pm\text{Id}$ then $\wt\rho$ has image in $\pm\text{Id}$ (as $\pi_1(E_K)$ is generated by conjugates of $m$). In such a case $\rho$ must be the trivial representation. Otherwise $\tr\wt\rho(m) = \tr(T) = 0$, and we have $\chi_{\wt\rho}\in X_{TF}(E_K)$, and $\Y(\chi_{\wt\rho}) = \chi_\rho$. 

By the discussion in Section \ref{SecBGCharVar}, particularly Remark \ref{rem-Tappliesmore}, the set $X^\tau(\dc K)$ is the intersection of $X(\dc K)$ with the set cut out by $\{\D(\rho(g_i),\rho(g_j),\rho(g_k))=0\}$ for a set of generators. So $X^\tau(\dc K)$ is algebraic.

\subsection{Characters from the cord ring}
\label{Sec-HCtocharofDC}

A map from the complexified abelian cord ring to $\bb C[X(\dc K)]$ can be defined as follows. An unoriented cord $\gamma\in\hc(K)$ has two lifts to $\dc K$, the union of which give, up to conjugation and inversion, an element $\wt\gamma\in\pi_1(\dc K)$. We define $\Z\co \hc(K)\otimes\bb C \to \bb C[X(\dc K)]$ by setting $\Z(\gamma\otimes z) = z t_{\wt\gamma}$ for $\gamma\otimes z\in\hc(K)\otimes\bb C$ (recall we use $t_{\wt\gamma}$ to denote the trace function of $\wt\gamma$). The well-definedness of $\Z$ follows from checking relations (1) and (2) of Definition \ref{defnAbCordRing}. Relation (1) holds since the identity has trace 2. That $\Z$ respects relation (2) is a consequence of the trace identity $\tr(AB)+\tr(AB^{-1}) = \tr(A)\tr(B)$ in $SL_2\bb C$. The map $\Z$ was introduced in \cite{Ng05b}, and was there shown to be an isomorphism if $K$ is a 2-bridge knot.


\begin{prop}Suppose that $K$ is the braid closure of $\beta\in B_N$ and define $\gamma_{ij},\ 1\le i<j\le N$ as in Section \ref{sec-HCvariety}. There exist $g_2,\ldots,g_N \in \pi_1(\dc K)$, that generate $\pi_1(\dc K)$, such that $\Z(\gamma_{ij}) = \bar t_{ij}$ for $1\le i< j\le N$, where $\bar t_{ij}$ is the trace function $t_{g_ig_j^{-1}}$ and we set $g_1$ to the identity.
\label{prop-surjectiononDCgroup}
\end{prop}

\begin{figure}[ht]
\begin{center}
\begin{tikzpicture}

\node[anchor=south west,inner sep=0] at (2,0) {\includegraphics[scale=0.5]{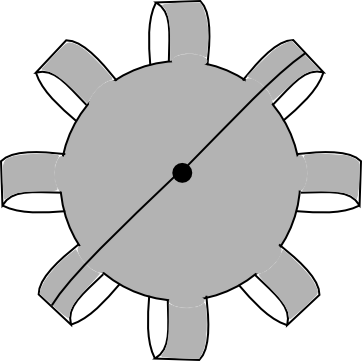}};


\draw (4.6,2.5) node[below] {$\wt p_1$};
\draw (5.5,3.5) node[below] {$f_j$};

\end{tikzpicture}

\end{center}
\caption{A page in the open book decomposition of $\dc K$ obtained from a representative of $\beta$}
\label{FigCoverofDn}
\end{figure}

\begin{proof}Choose a homeomorphism of $D_N$ fixing the boundary that represents $\beta$. This determines an open book decomposition of $\dc K$, where a page $S$ is a branched double cover of $D_N$. Identify $S$ with a disk with $N-1$ handles attached (see Figure \ref{FigCoverofDn}). As suggested by the figure we can put coordinates on the disk so that it is a unit disk in some plane in $\rls^3$ and the $k^{th}$ handle, $k=1,\ldots N-1$, has one foot attached to the disk on an interval about $\text{exp}(\sqrt{-1}\frac{k\pi}{N-1})$ and the other foot about $\text{exp}(\sqrt{-1}(\frac{k\pi}{N-1} + \pi))$. Label the center of the disk $\wt p_1$. We also require that rotation by $\pi$ about the normal line through $\wt p_1$ to this plane gives a self-homeomorphism of $S$ taking each handle to itself, each with one fixed point on the normal line (the fixed point orthogonally projecting to $\wt p_1$). This involution on $S$ has corresponding quotient map $q\co S\to q(S)$ is a branched double-cover over a disk with $N$ branch points, which we identify with $D_N$ by a homeomorphism that takes $q(\wt p_1)$ to $p_1$ and the image of the fixed point of the $k^{th}$ handle taken to $p_{k+1}$. Note that if $N$ is odd then $S$ has connected boundary and if $N$ is even, the two boundary components are interchanged by the involution. 

Write $f_j\in \pi_1(S,\wt p_1)$, $j=2,\ldots,n$ for the homotopy class of the union of an (equivariant) core of the $j-1^{th}$ handle with a line segment through $\text{exp}(\sqrt{-1}\frac{(j-1)\pi}{N-1})$ and $\text{exp}(\sqrt{-1}\frac{(j-1)\pi}{N-1}+\sqrt{-1}\pi)$ (with some choice of orientation). We may choose our identification $q(S)=D_N$ so that $q(f_j)$ is identified with $\gamma_{1j}$ for each $j=2,\ldots,N$.

Identify $S$ with a page in an open book decomposition of $\pi_1(\dc K)$ and let the basepoint of $\pi_1(\dc K)$ be the image of $\wt p_1$ under the inclusion map $S\hookrightarrow \dc K$, which induces a surjection on fundamental groups. If $g_j\in\pi_1(\dc K)$ is the image of $f_j$ then as $\wt \gamma_{1j} = g_j$ up to inversion and conjugation, we have $\Z(\gamma_{1j}) = t_{g_1g_j^{-1}}$. Moreover, for each $\wt \gamma_{ij}$ there is a path in the disk connecting $\wt \gamma_{ij}$ to $\wt p_1$ by which we obtain that $\wt \gamma_{ij}$, under the inclusion $S\subset\pi_1(\dc K)$ is homotopic (up to inversion) to $g_ig_j^{-1}$. Hence $\Z(\gamma_{ij}) = \bar t_{ij}$ also.
\end{proof}

Note that $t_it_j - \bar t_{ij} = t_{ij}$ and so the image of $\Z$ contains $\{t_i, 1\le i\le n\} \cup \{t_{ij}, 1< i<j\le n\}$. Also $\bar t_{ij} = \bar t_{ji}$ as the trace is invariant under inversion.

\begin{prop}Suppose $K$ is the closure of a braid $\beta\in B_N$ and $D(w(\beta))$ has $n$ crossings, with $N\le n$. There is a choice of meridians $\{m_1,\ldots,m_n\}$ generating $\pi_1(E_K)$ with an isomorphism $\pi_1(\dc K) \cong \wt\pi/\langle\langle m_1^2\rangle\rangle$ setting $[m_im_j^{-1}]$ equal to $g_ig_j^{-1}$ for $1\le i<j\le N$, where $[m_im_j^{-1}]$ is the class of $m_im_j^{-1}$ in $\wt\pi/\langle\langle m_1^2\rangle\rangle$. The $g_j, j=2,\ldots,N$ are as in the proof of Proposition \ref{prop-surjectiononDCgroup} and $g_1$ denotes the identity.
\label{prop-indextwo-to-dcK}
\end{prop}

\begin{figure}[ht]
\begin{center}
\begin{tikzpicture}

\node[anchor=south west,inner sep=0] at (2,0) {\includegraphics[scale=0.5]{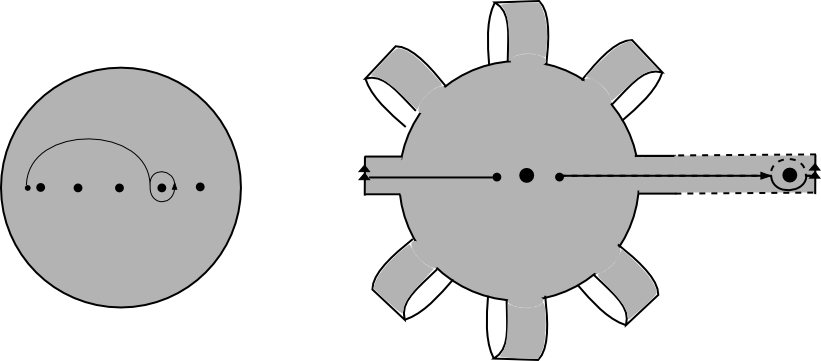}};


\draw (2.25,2.45) node {{\small $q$}};
\draw (2.6,2.4) node[below] {$p_1$};
\draw (3.15,2.4) node[below] {$p_2$};
\draw (3.75,2.4) node[below] {$p_3$};
\draw (4.3,2.4) node[below] {$p_4$};
\draw (4.9,2.4) node[below] {$p_5$};
\draw (3.45,3.1) node[above right] {$m_4$};

\draw (9.5,2.5) node[below] {$\wt p_1$};
\draw (9.9,2.5) node[below] {$\wt q$};
\draw (10.5,2.9) node {$\wt m_4$};
\draw (13.2,2.4) node[below] {$\wt p_4$};

\end{tikzpicture}

\end{center}
\caption{Left: meridian generators $m_i$ of $\pi_1(E_K,q)$; Right: a lift of $m_i$ to a page in $\dc K$}
\label{fig-appropriatemeridians}
\end{figure}
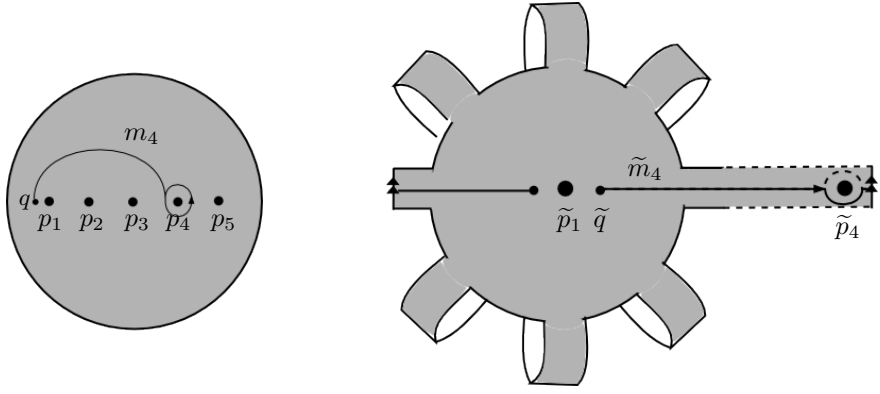

\begin{proof}For the basepoint of $\pi_1(E_K)$ take a point $q$ in $\bd(E_K) \cap D_N$ near $p_1$. For each $1\le i\le N$ take a push-off of $\gamma_{1i}$ a short distance toward the $z$-axis (a single point if $i=1$). Let $m_i$ be the concatenation of this push-off with a small, counterclockwise-oriented circle about $p_i$ (in $\bd(E_K)\cap D_N$ say), followed by the path traversing the push-off in reverse direction (see Figure \ref{fig-appropriatemeridians}). Abusing notation, we write $m_i$ both for this loop and the element it represents in $\pi_1(E_K,q)$.

The right part of Figure \ref{fig-appropriatemeridians} shows $S$, the branched double-cover over $D_N$, and the path to which $m_i$ lifts in $\dc K$ upon choosing a lift $\wt q$ of $q$ (the part of the handle, up to its fixed point, is displayed to the right side). Note that $m_i^2$ will lift to the (trivial) loop that follows the depicted path then traverses it in reverse order, except following the dashed arc near the fixed point on the handle upon its return. Fixing a short segment in $S$ from $\wt p_1$ to $\wt q$ we see that $[m_im_j^{-1}]$ corresponds to one of $(g_ig_j^{-1})^{\pm1}$ or $(g_ig_j)^{\pm1}$ and we can be careful about the choice of orientation in defining the $f_i$ in Proposition \ref{prop-surjectiononDCgroup} so that $[m_im_j^{-1}] = g_ig_j^{-1}$. Note that $[m_im_1^{-1}] = g_i^{-1} = g_1g_i^{-1}$ for $1<i\le N$.
\end{proof}

\begin{lem}A character $\chi\in X^\tau(\dc K)$ is determined by polynomials in its values on $g_ig_j^{-1}$, $1\le i,j\le N$ where $g_1$ is taken to be the identity.
\label{lem-CharDetermination}
\end{lem}

\begin{proof}We have seen already that $g_2,\ldots,g_N$ generate $\pi_1(\dc K)$, so by \cite{M}, as discussed in Section \ref{SecBGCharVar}, the coordinates of $\chi$ in $X^\tau(\dc K)$ are given by $t_i(\chi) = t_{1i}(\chi)$, $2\le i\le N$, $t_{ij}(\chi)=t_{g_ig_j}(\chi)$, $2\le i<j\le N$, and $t_{ijk}(\chi)$, $2\le i<j<k\le N$. 

Let $\hat\chi \in \Y^{-1}(\chi)$. Using the notation of Proposition \ref{prop-indextwo-to-dcK}, the value of $\bar t_{ij}$ on $\chi$ is given by $\hat\chi(m_im_j^{-1})$ and $t_{ij}$ is a polynomial in $\bar t_{1i},\bar t_{1j},$ and $\bar t_{ij}$. Now for $2\le i<j<k\le N$ we have $t_{ijk} = t_{ij}t_k - t_{g_ig_jg_k^{-1}}$. But $g_ig_jg_k^{-1} = [m_im_1^{-1}][m_jm_k^{-1}]$ and so $t_{g_ig_jg_k^{-1}}(\chi)$ is given by the trace function on $m_im_1^{-1}m_jm_k^{-1}$ evaluated on $\hat\chi$. But, consulting \cite[Lemma 4.1.1]{GM} for example, this is a polynomial in the $t_{m_im_j}$ on $\hat\chi$ {--} the three-index trace functions in this polynomial expression are all multiplied by zero since $\hat\chi$ is trace-free. 
\end{proof}

\section{Main Results}
\label{sec-results}

We are now ready to prove some of our main results. For a knot diagram of $D$, let $A^3(D)$ be the algebraic subset of $A(D)$ of augmentations with rank at most 3. Also, we denote the restriction of the dual $\Z^*\co X(\dc K)\to A(D)$ to $X^\tau(\dc K)$ by $\Z^\tau$. Recall the map $\Y$ from Section \ref{sec-knotgroup-to-dcK}.

\begin{thm}Let $D$ be a knot diagram representing $K$. There is a map $\X\co A^3(D) \to X_{TF}(E_K)$ such that $\Y\circ\X$ is a polynomial map and is the inverse of $\Z^\tau$. Thus $A^3(D)$ is equivalent to $X^\tau(\dc K)$.
\label{thm-varietycorrespondence}
\end{thm}

\begin{proof}
We may assume that $D=D(w(\beta))$ for some braid $\beta\in B_N$ which represents $K$. Say that $D$ has $n$ crossings. Let $\ve = (\ve_{12},\ve_{13},\ldots,\ve_{n-1,n})$ be a point of $A^3(D)$. If rank$(\ve) = 3$ then choose $1\le a_1<a_2<a_3\le N$ such that $\det\left[\ve_{ij}\right]_{i,j=a_1,a_2,a_3} = T({\bf a},{\bf a})\big|_{x_{ij} = \ve_{ij}}\ne 0$ and choose a square root of $\delta({\bf a})$ (we may restrict to the range $1,\ldots, N$ by Proposition \ref{prop-generatebyfewer}). By Proposition \ref{lem-SqrootCorrespondence}, using that no size four minor of $\left[\ve_{ij}\right]_{1\le i,j\le n}$ is non-zero, there is a unique $q\in X_{TF}(E_K)$ with $x_{a_1a_2a_3}(q) = \sqrt{\delta({\bf a})}$ and $x_{ij}(q) = \ve_{ij}$ for $1\le i<j\le n$. Take $q = \X(\ve)$. If rank$(\ve) < 3$ then $\X(\ve) = (\ve_{12},\ldots,\ve_{n-1,n},0,\ldots,0)$.

We will use the notation from Proposition \ref{prop-surjectiononDCgroup} and Proposition \ref{prop-indextwo-to-dcK}. Our definition of $\X$ gives coordinate functions $x_{ij}\circ\X(\ve) = \ve_{ij}$ for $1\le i< j\le n$. So associated to the character $\X(\ve)$ there is a trace-free representation $\rho\in R(E_K)$ so that $\tr(\rho(m_im_j^{-1})) = \ve_{ij}$. By Proposition \ref{prop-indextwo-to-dcK} we have $\Y\circ\X(\ve)$ evaluates to $\ve_{ij}$ on $[m_im_j^{-1}]=g_ig_j^{-1} \in \pi_1(\dc K)$ for any $1\le i<j\le N$ (recall that $g_1$ is the identity here). 

With coordinates $(\{t_i=t_{g_i},2\le i\le n\},\{\bar t_{ij}=t_{g_ig_j^{-1}},2\le i<j\le n\},\{t_{ijk}=t_{g_ig_jg_k},2\le i<j<k\le n\})$ on $X(\dc K)$, we get the $t_{i}$ coordinate of $\Y\circ\X(\ve)$ equal to $\ve_{1i}$ and the $\bar t_{ij}$ coordinate equal to $\ve_{ij}$. So by Lemma \ref{lem-CharDetermination} $\Y\circ\X$ is a polynomial map. For each $1\le i<j \le N$
		\[	\Z^\tau\left( \Y\circ\X(\ve) \right)(c_{ij}) = \bar t_{ij} (\Y\circ\X(\ve)) = \ve_{ij}. \]
By Proposition \ref{prop-generatebyfewer} this implies that $\Z^\tau\circ(\Y\circ\X)$ is the identity. To see that $(\Y\circ\X)\circ\Z^\tau$ is the identity, suppose the rank of $\Z^\tau(\chi)$ is at most three for $\chi\in X^\tau(\dc K)$. Since $\Z^\tau(\chi)(c_{ij}) = \bar t_{ij}(\chi)$, the discussion above says that $(\Y\circ\X)(\Z^\tau(\chi))$ takes value $\bar t_{ij}(\chi)$ on $g_ig_j^{-1}$, and so it agrees with $\chi$ on $g_ig_j^{-1}$, $1\le i<j\le N$. Lemma \ref{lem-CharDetermination} implies $(\Y\circ\X)(\Z^\tau(\chi)) = \chi$. We are left to see that $\text{rank}(\Z^\tau(\chi))\le 3$.

By Corollary \ref{cor-tracematrixisnotrank5} $C_D(\Z^\tau(\chi)) = \left[\bar t_{ij}(\chi)\right]$ can have no non-zero principal minor of size greater than four, hence $\Z^\tau(\chi)$ has rank at most four by Lemma \ref{lem-principalminors}. 

Using our generating set $g_2,\ldots,g_N$ for $\pi_1(\dc K)$, as in Proposition \ref{prop-indextwo-to-dcK}, for some $\rho$ with $\chi = \chi_\rho$ set $A_i=\rho(g_i)$. Since $\chi$ is $\tau$-equivariant there is a $T\in\SL_2\bb C$ so that $TA_iT^{-1} = \rho(m_1g_im_1^{-1})$ for each $i$ and by \cite[Lemma 19]{NY} either $T = \pm\text{Id}$ or $\tr(T) = 0$.

In the former case, $\chi$ is the image under $\Y$ of an abelian representation which means it is the character of the trivial representation. In the latter case, by Lemma \ref{lem-MagnusTraceless} either the trace of the commutators of pairs $A_i, A_j$ are each 2, or the Fricke discriminant $\D(A_i,A_j,A_k) = 0$ for each triple $1<i,j,k\le N$. If the traces of the all the commutators equal 2 then $\chi$ agrees with the character of an abelian representation, in which case $\D(A_i,A_j,A_k) = 0$ for each $i,j,k$ also. Now, setting $A_1$ to the identity, choose $k_1=1$ and $2\le k_i\le N, i=2,3,4$. Then
		\[\det\left[\bar t_{k_ik_j}(\chi)\right] = \det\left[\tr(A_{k_i}A_{k_j}^{-1})\right]\] 
which is zero since $\D(A_{k_2},A_{k_3},A_{k_4}) = 0$. This implies that each size 4 principal minor of $C_D(\Z^\tau(\chi))$ is zero which, as we already know the rank of $C_D(\Z^\tau(\chi))$ is at most four, implies in turn that $\Z^\tau(\chi)$ has rank three or less by Lemma \ref{lem-principalminors}.
\end{proof}

By Proposition \ref{prop-NYfixedpts} the involution $\iota\co X_{TF}(E_K)\to X_{TF}(E_K)$ has precisely the trace-free abelian character and the irreducible metabelian characters as fixed points. 

\begin{cor}Augmentations $\hc(K)\to\bb C$ with rank two are in bijective correspondence with irreducible metabelian characters of $\pi_1(E_K)$. Non-metabelian characters in $X_{TF}(E_K)$ are in 2-1 correspondence with the rank three augmentations of the abelian cord ring.
\label{thm-IrrMetabelian}
\end{cor}
\begin{proof}The definition of $\X$ makes it clear that $\X(\ve)$ is fixed under $\iota$ if and only if rank$(\ve) < 3$. (the coordinate $x_{ij}(\chi_\rho) = \tr(\rho(m_im_j^{-1})) = \tr(\iota(\rho)(m_im_j^{-1}))$, but $x_{ijk}(\chi_\rho) = \tr(\rho(m_im_jm_k)) = -\tr(\iota(\rho)(m_im_jm_k))$)

As $\X$ is injective and surjective on orbits of $\iota$ (since $\Y\circ\X$ is onto $X^\tau(\dc K)$, and the rank 1 reflective augmentation ($\ve(c_{ij}) = 2$ for all $i,j$) gives the abelian representation, we see that $\X$ sets up a 1-1 correspondence between rank 2 augmentations in $A(D)$ and irreducible metabelian characters in $X_{TF}(E_K)$. Moreover, $\X$ sends a rank three augmentation to a non-metabelian character and $\Z^\tau\circ\Y$ is 2-1 on non-metabelian characters by Theorem \ref{thm-NYInvolution} and that $\Z^\tau$ is injective on $X^\tau(\dc K)$.
\end{proof}

\begin{cor}Let $K$ be a knot with $|\D_K(-1)| = 1$. Then $K$ has a rank 3 augmentation $\ve\co \hc(K)\to\bb C$ and no rank 2 augmentations.
\label{cor-Det1}
\end{cor}
\begin{proof}By \cite{KM} there is an irreducible representation $\rho\co \pi_1(E_K)\to SU(2)$, such that $\rho(m)$ is a diagonal matrix with eigenvalues $\pm\sqrt{-1}$, which is not binary-dihedral as the determinant is 1. Since $\chi_\rho\in X_{TF}(E_K)$, Theorem \ref{thm-IrrMetabelian} implies there is a rank 3 augmentation and additionally no rank 2 augmentations since there are no irreducible metabelian characters.
\end{proof}
\begin{cor}Let $U$ be the unknot in $S^3$. If $\hc(K)\cong\hc(U)$ then $K=U$.
\label{cor-Udetection}
\end{cor}
\begin{proof}Suppose $K$ is not the unknot. By the result of \cite{KM} cited in Corollary \ref{cor-Det1}, for any diagram $D$ of $K$, it follows from Corollary \ref{thm-IrrMetabelian} that $A^3(D)$ at least two points: one corresponding to the unique rank 1 augmentation and another for the irreducible $\SU(2)$ character. So there cannot be an isomorphism $\hc(K)\cong\hc(U)$ as $A(U)$ is a single point.
\end{proof}
There are distinct 2-bridge knots with isomorphic cord rings by \cite[Theorem 4.3]{Ng05b}. We also have the following.
\begin{thm}Suppose that $\text{max}\ \mathfrak r^{ab}(K)\le 3$. Then $X^\tau(\dc K) = X(\dc K)$ and $(\hc(K)\otimes\bb C)/\sqrt0 \cong \bb C[X(\dc K)]$.
\label{thm-equivariant}
\end{thm}
\begin{proof}Each abelian representation is in the image of $\Y$, its preimage a metabelian representation, and so by \cite{NY} is in $X^\tau(\dc K)$ and the character of any reducible representation agrees with the character of an abelian representation.

Now suppose $\chi\in X(\dc K)$ is associated to an irreducible representation $\rho\co \pi_1(\dc K)\to\SL_2\bb C$. Consider a generating set $g_2,\ldots,g_N$ for $\pi_1(\dc K)$ as above and let $A_i = \rho(g_i)$ for $1< i \le N$.

A fact we have used before is that for $A_2,A_3,A_4\in\SL_2\bb C$ the discriminant $\D(A_2,A_3,A_4)$ is zero if and only if $\det[\tr(A_iA_j^{-1})]_{1\le i,j\le 4}=0$, where $A_1$ is the identity (see comments after Lemma \ref{lem-MagnusTraceless}). For each $i,j$ the trace $\tr(A_iA_j^{-1})$ is the evaluation of $\Z(c_{ij})$ on $\chi_\rho$, and so $\left|\tr(A_iA_j^{-1})\right|$ is $\Z^\tau(\chi_\rho)\left(\det[c_{ij}]_{1\le i,j\le 4}\right) = 0$, since we assumed no reflective augmentation with rank more than three exists.

Thus by Lemma \ref{lem-MagnusTraceless} there is a $T\in\SL_2\bb C$ with $\tr(T) = 0$ such that $(TA_i)^2 = -\text{Id}$ for $i=2,3,4$, provided the trace of $[A_i,A_j]$ is not 2 for some pair $i,j$, which we may assume is true, for $i=2,j=3$ say, as $\rho$ is irreducible \cite{CS}. By Remark \ref{rem-Tappliesmore}, since $\D(A_2,A_3,A_j)=0$ for $4\le j\le n$ (for the same reason), we have $(TA_j)^2=-\text{Id}$ for $4\le j\le n$ also.

Recall $m_1m_i^{-1}\in \wt\pi$ satisfies $[m_1m_i] = [m_1m_i^{-1}]=g_i\in\pi_1(\dc K)$, so we have for any $2\le i\le N$,
			\[\rho([m_1m_1m_i^{-1}m_1^{-1}]) = \rho([m_1m_i])^{-1} = \rho(g_i)^{-1} = -(T\rho(g_i))^{-1}T^{-1} = T\rho(g_i)T^{-1}.\]
As the $g_i$ generate this means $\rho\in R^\tau(\dc K)$, so $\chi\in X^\tau(\dc K)$.

Thus $\text{max}\ \mathfrak r^{ab}(K)\le 3$ implies that $X^\tau(\dc K) = X(\dc K)$. We then obtain the isomorphism $\hc(K)\otimes\bb C / \sqrt0 \cong \bb C[X(\dc K)]$ from Theorem \ref{thm-varietycorrespondence} and the Nullstellensatz.
\end{proof}

When $K$ has reflective augmentations of sufficiently high rank we can see that $\Z\co(\hc(K)\otimes\bb C)/\sqrt0 \to \bb C[X(\dc K)]$ will not be an isomorphism.

\begin{prop}Suppose that $K$ has a reflective augmentation of rank at least five. Then there is $b\not\in \sqrt0$ with $b\in\ker\Z$, so $\Z\co \hc(K)\otimes\bb C/\sqrt0\to\bb C[X(\dc K)]$ is not an isomorphism.
\label{prop-NontrivialKernel}
\end{prop}

\begin{proof}That the dual map $\Z^\tau$ cannot be an equivalence was already observed in the proof of Theorem \ref{thm-varietycorrespondence} {--} the observation there that $\Z^\tau(\chi)$ could not have rank more than four did not depend on $\chi$ being $\tau$-equivariant. In order to see an element of the kernel of $\Z$, choose an augmentation $\ve$ with rank five, and a size 5 non-degenerate principal submatrix of $C_D(\ve)$. The determinant of the corresponding submatrix of $C_D$ (as a matrix over $\hc(K)$) cannot be in the nilradical of $\hc(K)$ because of $\ve$, but it is in the kernel of $\Z$ by Corollary \ref{cor-tracematrixisnotrank5}.
\end{proof}

We might then expect that, given the existence of higher rank augmentations, the rings $(\hc(K)\otimes\bb C)/\sqrt0$ and $\bb C[X(\dc K)]$ are not isomorphic.

\section{Applications, finding representations into $\SU(2)$ and $\SL_2\bb R$}
\label{sec-unitary}

In this section we explicitly give a non-metabelian trace-free representation $\wt\rho\co \pi_1(E_K)\to\SL_2\bb C$ corresponding to $\ve\in A(D)$, with rank three. The obvious necessary conditions on $\ve$ for $\wt\rho$ to be conjugate into $\SU(2)$ are shown to be sufficient, and that conjugate representation is non-binary-dihedral. The representation $\wt\rho$ is lifted from a non-abelian representation $\rho\co \pi_1(\dc K)\to \SL_2\bb C$ with character $\Y(\chi_{\wt\rho})$. We discuss when $\rho$ is conjugate into $\SL_2\bb R$and give some applications to the left-orderability of $\pi_1(\dc K)$.

\subsection{Constructing representations}
\label{sec-makingReps}

Suppose that $D$ is $D(w(\beta))$ for some word for a braid representative $\beta\in B_N$ of $K$ (we will note shortly that, in fact, $D$ may be any knot diagram of $K$). Let $\ve = (\ve_{12},\ve_{13},\ldots,\ve_{n-1,n})\in A(D)$ have rank three. By Lemma \ref{lem-principalminors} there is a size 3 principal submatrix of $C_D(\ve)$ that is non-degenerate. We may appropriately label $D$ so that it consists of rows (and columns) $1,2,3$; note this implies
		\[\ve_{12}^2+\ve_{13}^2+\ve_{23}^2 - \ve_{12}\ve_{13}\ve_{23}-2 \ne 2.\] 
We may also assume $\ve_{12}\ne\pm2$ (Lemma \ref{lem-rank1aug}). Set $A_1=\text{Id}$ and construct $A_1,\ldots, A_n$ as in Proposition \ref{prop-augtotrace}. We use the notation of that proposition (taking $k=3$), so that
	\[A_2 = \begin{pmatrix}d & 0\\ 0 & \ve_{12} - d\end{pmatrix},\ A_3 =\begin{pmatrix}a & 1 \\ a(\ve_{13} - a) -1 & \ve_{13} - a\end{pmatrix},\ \text{and }A_l = \begin{pmatrix}a_l & b_l \\ c_l & \ve_{1l} - a_l\end{pmatrix}\] for every $l\ne 1,2,3$. By that proposition we have $A_1,\ldots,A_n$ defined such that $\tr(A_rA_s^{-1})=\ve_{rs}$ if $\min(r,s) \le 3$.
We are left to check that $\tr(A_rA_{s}^{-1}) = \ve_{rs}$ if $r$ and $s$ are larger than 3. To this end consider the $4\times 4$ matrix 
		\[M(r,s) = \begin{pmatrix} 2 & \tr(A_2^{-1}) & \tr(A_3^{-1}) & \tr(A_{s}^{-1}) \\ \tr(A_2) & 2 & \tr(A_2A_3^{-1}) & \tr(A_2A_{s}^{-1}) \\ \tr(A_3) & \tr(A_3A_2^{-1}) & 2 & \tr(A_3A_{s}^{-1}) \\ \tr(A_r) & \tr(A_rA_2^{-1}) & \tr(A_rA_3^{-1}) & \tr(A_rA_s^{-1}) \end{pmatrix}.\]
For any $r$ it must be that $\det(M(r,r)) = 0$ since each entry of $M(r,r)$ equals the corresponding augmentation coordinate $\ve_{\ast,\ast}$ and the rank of $\ve$ is less than four. By Corollary \ref{cor-switchcolumn} it must be that $\det(M(r,s)) = 0$ for $r, s\ne 1,2,3$. As the rank of $\ve$ is three, our choice of labeling implies the fourth column of $M(r,s)$ is in the linear span of the (independent) first three columns. Since $\tr(A_lA_s^{-1})=\ve_{ls}$ for $l=1,2,3$ this also implies $\tr(A_rA_s^{-1}) = \ve_{rs}$.

Let $m_1,\ldots,m_n$ be the meridian generators of $\pi_1(E_K)$ associated to $D$. Take $g_i=[m_1m_i^{-1}]$ for $i=2,\ldots,n$ (so $g_2,\ldots,g_N$ are as in Proposition \ref{prop-indextwo-to-dcK}) and define $\rho(g_i)=A_i$ for each $i$, defining a map on $\pi_1(\dc K)$. Since $\Y\circ\X(\ve)\in X(\dc K)$ is a well-defined character, there is some representation $\rho'$ of $\pi_1(\dc K)$ in $\SL_2\bb C$ satisfying the above trace conditions. But up to conjugacy the trace conditions determined $\rho$, i.e.\ there exists $B\in\SL_2\bb C$ so that $B\rho'(g_i)B^{-1} = A_i$, $i=2,\ldots,n$, so $\rho$ must be well-defined. 

To construct its lift $\wt\rho$, a trace-free representation $R(E_K)$, construct the matrix $T$ guaranteed by Lemma \ref{lem-MagnusTraceless} (that $M(r,r)=0$ implies $\D(A_2,A_3,A_r)=0$). Define $\wt\rho(m_i) = TA_i$ for each meridian generator $m_i,\ r=1,2,\ldots,n$. Then $\wt\rho\co \pi_1(E_K)\to\SL_2\bb C$ has $\X(\ve)$ as its character and is a well-defined, trace-free representation that is non-metabelian as $\ve$ had rank three (see comments following Theorem \ref{thm-NYInvolution}).

Define $\alpha := 1-a(\ve_{13}-a)$. We will need the following lemma in the remaining sections.

\begin{lem}If $\ve\in A(D)$ is a real point, i.e.\ $\ve_{rs}\in\bb R$ for $1\le r<s \le n$, then $\alpha\in\bb R$. In that case $\alpha$ is positive if and only if $\ve_{12}^2 - 4$ and $\ve_{12}^2+\ve_{13}^2+\ve_{23}^2-\ve_{12}\ve_{13}\ve_{23} - 4$ have the same sign.
\label{lem-alpha}
\end{lem}
\begin{proof}Using the definition of $a$, the conclusion is a direct consequence of a computation:
	\al{
	\alpha = a^2 - a\ve_{13} + 1	&= \frac{(\ve_{23}-\ve_{13}d)^2 - \ve_{13}(\ve_{23} - \ve_{13}d)(\ve_{12} - 2d)}{\ve_{12}^2 - 4} + 1 \\
									&= \frac{\ve_{12}^2 + \ve_{13}^2+\ve_{23}^2 - \ve_{13}\ve_{23}\ve_{12}-4}{\ve_{12}^2 - 4}.
	}
\end{proof}

\begin{rem}
We have been using a closed braid form and the diagram $D(w(\beta))$ because of our choice of presentation of $\pi_1(\dc K)$. However, if $D$ is any knot diagram, and $m_1,\ldots,m_n$ the set of meridian generators in the Wirtinger presentation, it is possible to directly check that setting $m_1\mapsto T$ and $m_i\mapsto TA_i$, for $i=2,\ldots,n$ determines a well-defined representation of $\pi_1(E_K)$. 

Suppose that there is a crossing $(i,j,k)$ in $D$ (recall the overcrossing strand is $i$). Then the corresponding relation in the Wirtinger presentation is $m_jm_i^{\epsilon} = m_i^{\epsilon}m_k$, where $\epsilon$ is either 1 or -1. Moreover, $\ve_{rj}+\ve_{rk} = \ve_{ri}\ve_{ij}$ for any $1\le r\le n$. One can check this implies that $TA_j+TA_k = \ve_{ij}TA_i$. So if $\ve_{ij}\ne 0$ then, since $(TA_r)^2=-\text{Id}$ for $r= j,k$,
				\[(TA_j)(TA_i) = \frac1{\ve_{ij}}((TA_j)(TA_k) - \text{Id}) = (TA_i)(TA_k),\]
so the relation holds. The relation also holds if $\ve_{ij}=0$, but takes some more effort to check. In particular, it requires one to use the fact that $c_l+b_l\alpha = 0$ (see the proof of Theorem \ref{thm-unitary2}). 
\end{rem}

\subsection{$\SU(2)$ representations}
\label{sec-SU2}
Before discussing the relation to $\SU(2)$ characters in $X_{TF}(E_K)$, we point out a relationship to a space arising in singular instanton knot homology $I^\nl(K)$ \cite{KMunknotdetector}. The conjugacy classes of $\SU(2)$ representations in $X_{TF}(E_K)$ is a space arising in the construction of $I^\nl(K)$.

Thinking of $\SU(2)$ as the unit quaternions, define $R(K;\bi) = \{\rho\co \pi_1(E_K)\to\SU(2) | \rho(m)\sim\bi \}$, where $m$ is a meridian and $\rho(m)\sim\bi$ means that $\rho(m)$ is conjugate to $\bi$. Identifying $\SU(2)$ with a real subalgebra of $M_2(\bb C)$, setting $\bi$ to the diagonal matrix $\text{diag}[\sqrt{-1},-\sqrt{-1}]$, we note that $R(K;\bi)$ is in the trace-free representations of $\pi_1(E_K)$. The binary-dihedral subgroup is $S^1\cup S^1{\bf j}\subset \SU(2)$, where $S^1 = \{a+b\bi | a^2+b^2 = 1\}$. The set of representations with binary-dihedral image are all in $R(K,\bi)$ and their characters are the metabelian characters. 
\begin{defn}A knot $K$ is \textit{$\SU(2)$-simple} if $R(K;\bi)$ contains only binary-dihedral representations.
\end{defn}

By results below, the rank three reflective augmentations determine whether a knot is $\SU(2)$-simple. If a certain genericity hypothesis holds for $K$ and $K$ is $\SU(2)$-simple then the singular instanton chain complex, under an admissible holonomy perturbation, has no non-zero differentials and the total rank of $I^\nl(K)$ is $\det(K)$ \cite[Proposition 7.3]{Zentner} (see also \cite{hed-her-kirk}).

Returning to constructions on augmentations, let $\wt\rho\co \pi_1(E_K)\to\SL_2\bb C$ be the representation constructed in Section \ref{sec-makingReps}. For a pair $m_r,m_s$ of meridians we have $\ve_{rs} = \tr(\wt\rho(m_rm_s^{-1}))$. Thus if $\wt\rho$ is conjugate into $\SU(2)$ then $\ve_{rs}\in[-2,2]\subset\bb R$. Since the commutator $[\wt\rho(m_sm_r^{-1}), \wt\rho(m_rm_t^{-1})]$ has trace $\ve(r,s,t) := \ve_{rs}^2 + \ve_{rt}^2 + \ve_{st}^2 - \ve_{rs}\ve_{rt}\ve_{st} - 2$, we have $\ve(r,s,t)\in[-2,2]$ also. 

\begin{defn}Let $\ve\in A(D)$ have rank three for $D$ an $n$-crossing diagram. We call $\ve$ \textit{elliptic} if it satisfies the conditions
			\al{
			\ve_{rs} \in [-2,2],\quad 				& 1\le r < s \le n \\
			\ve(r,s,t) \in [-2,2],\quad 			& 1\le r < s < t \le n.
			}
\label{defn-elliptic}
\end{defn}
Note that if $\ve$ is elliptic then there is some $\ve_{rs}\in(-2,2)$, since $\text{rank}(\ve)=3$. So the images of $\rho$ and $\wt\rho$ contain some elliptic element in $\SL_2\bb C$, and this elliptic element is involved in a commutator that is also elliptic.

\begin{thm}Given a diagram $D$ and $\ve\in A(D)$, $\ve$ is elliptic if and only if $\X(\ve)$ is an $\SU(2)$ character.
\label{thm-unitary2}
\end{thm}

\begin{proof}
If $\X(\ve)$ is an $\SU(2)$ character then $\ve$ is elliptic. Suppose that $\ve$ has rank three and is elliptic. 

Let $\bi = \sqrt{-1}$ and construct $T$ and $A_2,\ldots,A_n$ in $\SL_2\bb C$ as at the start of the section. As $\ve$ has rank three, we may suppose $\ve_{12} \in (-2,2)$ and $\ve(1,2,3) < 2$. Use the notation introduced in Proposition \ref{prop-augtotrace}. Recall $a = \frac{\ve_{23} - \ve_{13}d}{\ve_{12}-2d}$ and $d^2-\ve_{12}d + 1 = 0$.

Choose a square root of $\alpha = 1-a(\ve_{13}-a)$. By Lemma \ref{lem-alpha} $\alpha\ne0$ and by the explicit method of constructing $T$ (Remark \ref{rem-Tappliesmore}) we may set

			\[T = \begin{pmatrix} 0 & \bi\alpha^{-1/2} \\ \bi\alpha^{1/2} & 0\end{pmatrix},\]
which has eigenvectors $(\alpha^{-1/2},1)$ and $(\alpha^{-1/2},-1)$. As $\ve_{12} \in (-2,2)$ we get $d^2-\ve_{12}d+1 = 0$ implies that $\ve_{12} - d = \bar d$, the complex conjugate of $d$. Letting $P$ be the matrix with the above eigenvectors as columns, we have 
			\[P^{-1}TA_2P = \begin{pmatrix}\frac{\bi}2(d+\bar d) & \frac{\bi}2(d-\bar d) \\ -\frac{\bi}2(d-\bar d) & -\frac{\bi}2(d + \bar d)\end{pmatrix}\]
which is in $\SU(2)$. Also we have that
			\[P^{-1}TA_3P = \begin{pmatrix}\frac {\bi}2\ve_{13} & -\frac {\bi}2(2\alpha^{1/2} + (\ve_{13} - 2a)) \\ -\frac {\bi}2(2\alpha^{1/2} - (\ve_{13} - 2a)) & -\frac {\bi}2\ve_{13}\end{pmatrix}.\]
As $\ve(1,2,3) < 2$ we know $\alpha^{1/2}$ is real by Lemma \ref{lem-alpha}. Moreover, since $\ve_{12}-d = \bar d$ one can show that the real part of $2a$ is $\ve_{13}$. Hence $\ve_{13} - 2a$ is purely imaginary. From these observations we see that $P^{-1}TA_3P$ has the form $\begin{pmatrix}\alpha & \beta \\ -\bar\beta & \bar\alpha\end{pmatrix}$.

Finally consider $A_l$ for any $l\ne 1,2,3$. Recall that $c_l = a(\ve_{1l} - a_l) + b_l - b_la(\ve_{13}-a) + (\ve_{13} - a)a_l - \ve_{3l}$, where $a_l = \frac{\ve_{2l} - \ve_{1l}d}{\ve_{12} - 2d}$ and $b_l$ is defined such that the determinant is 1.

Straightforward computation gives that
			\[P^{-1}TA_lP = \frac{\alpha^{-1/2}}2\begin{pmatrix}\bi\ve_{1l}\alpha^{1/2} + \bi(c_l + b_l\alpha) & \bi(2a_l - \ve_{1l})\alpha^{1/2} + \bi(c_l - b_l\alpha) \\ \bi(\ve_{1l} - 2a_l)\alpha^{1/2} + \bi(c_l - b_l\alpha) & -\bi\ve_{1l}\alpha^{1/2} + \bi(c_l + b_l\alpha)\end{pmatrix}\]
Now note that $2a_l - \ve_{1l} \in \bi\bb R$. Moreover,
		\al{
			c_l - b_l\alpha 	&= a(\ve_{1l} - a_l) + a_l(\ve_{13} - a) - \ve_{3l} \\
						&= \frac2{4-\ve_{12}^2}(\ve_{23}(\ve_{2l} - \frac{\ve_{12}}2\ve_{1l}) + \ve_{13}(\ve_{1l} - \frac{\ve_{12}}2\ve_{2l})) - \ve_{3l},
		}
and so $c_l-b_l\alpha\in\bb R$ and thus the off-diagonal entries in $P^{-1}TA_lP$ make a pair $\beta$ and $-\bar\beta$. Now we note that $c_l + b_l\alpha$ must be zero, as $P^{-1}TA_lP$ has trace zero. So $P^{-1}TA_lP$ is in $\SU(2)$. 

To see $c_l + b_l\alpha=0$ more directly, since $P^{-1}TA_lP$ has determinant 1 we have
		\[1 + \frac14\left((2a_l - \ve_{1l})^2 - \alpha^{-1}(c_l-b_l\alpha)^2\right) = \frac14\left(\ve_{1l}^2-\alpha^{-1}(c_l+b_l\alpha)^2\right),\]
and rearranging gives $(c_l+b_l\alpha)^2 = \alpha(\ve_{1l}^2-4) - \alpha(2a_l - \ve_{1l})^2 + (c_l-b_l\alpha)^2$. Now consider the determinant, which we will write $\det_{l}(\ve)$ of the $4\times 4$ matrix $\left[\ve_{rs}\right]_{r,s=1,2,3,l}$. Direct comparison finds that
	\begin{equation}
	(c_l - b_l\alpha)^2 + \alpha(\ve_{1l}^2 - 4) - \alpha(2a_l - \ve_{1l})^2 = \alpha\frac{\det_{l}(\ve)}{\ve_{12}^2+\ve_{13}^2+\ve_{23}^2 - \ve_{12}\ve_{13}\ve_{23} - 4}.
	\label{eqn-magic}
	\end{equation}
As $\ve$ has rank three, $\det_{l}(\ve) = 0$, and so $c_l + b_l\alpha = 0$.
\end{proof}


\begin{rem}Note in the proof of Theorem \ref{thm-unitary2}, we only needed to use that $\ve$ was a real point in $A(D)$ and that for \textit{some} $1\le i<j<k\le n$ (relabelled to be $i=1,j=2,k=3$) we had $\ve_{ij} \in (-2,2)$ and $\ve(i,j,k) < 2$. Yet the result implies that $\ve$ is elliptic, placing a strong condition on the real points in $\bb C^{\binom n2}$ which can be occupied by an augmentation.
\label{rem-Realpoints}
\end{rem}

\begin{cor}A knot $K$ is $\SU(2)$-simple if and only if $K$ has no elliptic augmentation.
\label{cor-su2simple}
\end{cor}
\begin{proof}
If $K$ has a elliptic augmentation $\ve$ then by Theorem \ref{thm-unitary2} $\X(\ve)$ is the character of a $\SU(2)$ representation $\rho\co\pi_1(E_K)\to\SU(2)$. That representation is trace-free, so $\rho\in R(K;\bi)$ and since $\ve$ has rank three, $\rho$ is not binary-dihedral by Corollary \ref{thm-IrrMetabelian}. So $K$ is not $\SU(2)$-simple.

If $K$ is not $\SU(2)$-simple then there is an irreducible $\SU(2)$ representation in $X_{TF}(E_K)$ that does not have a metabelian character. If the character is $\chi$, then $\Z^\tau\circ\Y(\chi)\in A^3(D)$ has rank three and is elliptic since $\chi$ is an $\SU(2)$ character. 
\end{proof}

In the appendix we record the number of elliptic augmentations for knots with 8, 9, and 10 crossings which are not 2-bridge knots (2-bridge knots have no rank three augmentations). It should be noted that each non-alternating knot in these tables has at least one elliptic augmentation. The author has also looked at the reflective augmentations of all non-alternating 11-crossing knots, each of which has at least one elliptic augmentation. This gives some computational evidence for an affirmative answer to Question 8.2 in \cite{Zentner}. For the reader's convenience we note that from Tables \ref{table-8crossing}, \ref{table-10crossingnon}, and \ref{table-10crossingalt} the following is the set of knots with up to 10 crossings which are $\SU(2)$-simple, but not 2-bridge:
		\[8_{16},8_{17},8_{18},9_{29},9_{32},9_{33},9_{38},10_{83},10_{86},10_{88},10_{89},10_{91},10_{92},10_{94},10_{95},10_{96},10_{97},10_{109}.\]

While the author is uncertain how generally it holds, the phenomenon illustrated by the Example \ref{example-2} below may be of interest to the question of whether any non-alternating knot is $\SU(2)$-simple.

\begin{example}

Consider the knot $10_{153}$ that appears in the knot tables and its diagram $D_{153}$ shown in Figure \ref{fig-knotdiagrams}. The diagram can be cut at 4 points (as in the figure) to obtain two alternating tangles. Note that if we ``close'' the tangle outside of the circle by one arc attaching the components 5 and 6 and another arc attaching components 1 and 7 we obtain the labeled diagram $D_0$ of $K_0=5_2$ in Figure \ref{fig-firstexample}. 

Recall from Example \ref{example-1} that we determine an augmentation of $\hc(K_0)$ by sending $c_{12}$ to a root of $f(x):=x^3+x^2-2x-1$ (each of these augmentations has rank 2). We also note that $c_{14}=c_{12}^3-3c_{12}$ and $c_{15}=c_{12}^2-2$ are sent by such an augmentation to the other two roots.

There are 5 rank three augmentations in $A(D_{153})$. For each root $x_i$, $i=1,2,3$ of $f(x)$, one of these augmentations sends $c_{12}$ to $x_i$. Note that all three roots are real. Let $x_1,x_2,x_3$ be the roots in order, i.e.\ $x_1<x_2<x_3$. Denote by $\ve_i$ the augmentation with $\ve_i(c_{12})=x_i$. The matrix $C_{D_{153}}(\ve_2)$ is given below (above the diagonal), where $y_i, i=1,2,3$ denote the three real roots (in order) of $g(x)=x^3-2x^2-x+1$. The augmentations $\ve_1$ and $\ve_3$ are similarly described but under a cyclic permutation applied to the roots (of both $f$ and $g$).
{\small
\[\begin{pmatrix}2&x_2&x_2&x_3&x_1&-x_1&-2&1&1&-1 \\ &2&x_1&x_3&x_3&-x_3&-x_2&-y_3&-x_1&x_1 \\ & &2&x_1&x_2&-x_2&-x_2&-x_1&-y_3&y_3 \\ & & &2&x_2&-x_2&-x_3&-1&y_3&-y_3 \\ & & & &2&-2&-x_1&x_1&0&0 \\ & & & & &2&x_1&-x_1&0&0 \\ & & & & & &2&-1&-1&1 \\ & & & & & & &2&-1&1 \\ & & & & & & & &2&-2 \\ & & & & & & & & &2
\end{pmatrix}\]
}

We remark that $x_i\in[-2,2]$ for all $i$, $y_1,y_2\in[-2,2]$ and $y_3>2$. Indeed, $\ve_1$ and $\ve_3$ are elliptic augmentations and $\ve_2$ is a non-elliptic rank three augmentation.

We also note that, focusing on cords involving only components in the tangle giving $K_0$ (i.e.\ $\ve_{ij}$ for $1\le i,j\le 5$), we can check the augmentation in Example \ref{example-1} to see this $5\times 5$ submatrix equals $C_{D_0}(\ve_0)$. If we consider the other tangle in Figure \ref{fig-knotdiagrams} and close it by attaching components 1 and 5, and components 6 and 7, we obtain a diagram of $5_1$. We said that $A(D_{153})$ has 5 rank three augmentations. It seems that the other 2 augmentations have a similar relationship to the two rank 2 augmentations of $5_1$.
\label{example-2}
\end{example}

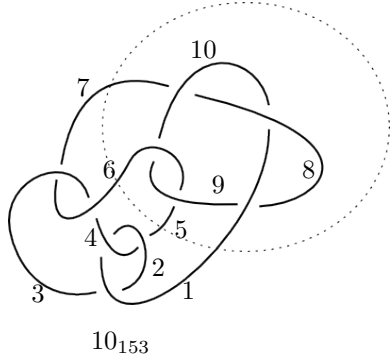
\begin{figure}[ht]
\begin{center}
\begin{tikzpicture}

\node[anchor=south west,inner sep=0] at (0,1) {\includegraphics[scale=0.5]{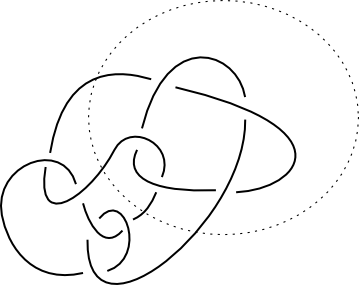}};

\draw(1.5,0.5) node {$10_{153}$};

\draw(4,2.8) node {$8$};
\draw(2.3,2) node {$5$};
\draw(0.4,1.15) node {$3$};
\draw(1.1,1.9) node {$4$};
\draw(2.4,1.2) node {$1$};
\draw(2,1.5) node {$2$};
\draw(1,3.9) node {$7$};
\draw(1.35,2.8) node {$6$};
\draw(2.8,2.6) node {$9$};
\draw(2.6,4.4) node {$10$};
\end{tikzpicture}
\end{center}
\caption{Diagram of the determinant one knot $10_{153}$.}
\label{fig-knotdiagrams}
\end{figure}

\subsection{$\SL_2\bb R$ representations of $\pi_1(\dc K)$ and orderability}
\label{sec-sl2R}
In contrast to Section \ref{sec-SU2} suppose that $\ve\in A(D)$ is a real point, with rank three, which is not elliptic. Then it turns out that $\Y\circ\X(\ve)=\chi_\rho$, where $\rho$ is constructed as above, is the character of a representation into $\SL_2\bb R$.

Note that $\chi_\rho\in X^\tau(\dc K)$ is a real-valued character by Lemma \ref{lem-CharDetermination} and the fact that $\ve$ is real. (This may not be the case for $\chi_{\wt\rho}$.) But if a real-valued character is not the character of an $\SL_2\bb R$ representation then it is an $SU(2)$ character (see e.g.\ \cite[Proposition III.1.1]{MS}). But $\ve$ is not elliptic so this is not the case. Since $\chi_\rho$ is not an abelian character, $\rho$ is irreducible and so conjugate to an $SL_2\bb R$ representation. Often the representation $\rho$ itself is into $\SL_2\bb R$. 

\begin{thm}
Given a rank three real point $\ve\in A(D)$, if there exists $1\le i<j<k\le n$ with $|\ve_{ij}|>2$ and $\ve(i,j,k)>2$ then the non-abelian representation $\rho\co\pi_1(\dc K)\to\SL_2\bb C$, defined by $\rho(g_r) = A_r$ as above, has image in $\SL_2\bb R$.
\label{thm-SL2R}
\end{thm}
\begin{proof}  
Use the notation from Section \ref{sec-makingReps}, assuming $i=1,j=2,k=3$. The fact $\ve_{12}^2 > 4$ implies that $d\in \bb R$, which ensures that $A_2$ and $A_3$ are in $\SL_2\bb R$ since $\ve$ is real. For $l\ne 1,2,3$ it is then sufficient to check that $b_l \in\bb R$.

Recall Equation (\ref{eqn-magic}) which, since $\ve$ has rank three, implies that $(c_l-b_l\alpha)^2 = \alpha\left((2a_l-\ve_{1l})^2 + 4 - \ve_{1l}^2\right)$. Now the discriminant of the (quadratic) equation $\det(A_l) = 1$ used to define $b_l$, is
	\begin{equation}(a(a_l - \ve_{1l}) + a_l(a - \ve_{1k}) + \ve_{3l})^2 + 4\alpha(a_l(\ve_{1l} - a_l)) = (c_l - b_l\alpha)^2+ 4\alpha(a_l(\ve_{1l} - a_l)).
	\label{eqn-discriminant}
	\end{equation}
So by (\ref{eqn-magic}) this discriminant is $4\alpha$, which is a real positive number (that is independent of $l$) by Lemma \ref{lem-alpha}. So $b_l\in \bb R$.
\end{proof}

Consequently augmentations can provide information about orderings on $\pi_1(\dc K)$. Two types of orderings are of interest. A left-invariant order (or simply \textit{left ordering}) on a group $G$ is, as it sounds, a total ordering which is preserved by multiplication on the left. A \textit{circular ordering} on $G$ is one that is modelled by an oriented $S^1$. That is, one has for each $g\in G$ a total order $<_g$ on $G\setminus\{g\}$ with a compatibility condition on two such orders: if $g,h,x,y\in G$ are pairwise distinct elements and $x <_g y$ then either $x <_h y$ or (exclusive) $x <_g h <_g y$. 

It is known that if a compact, connected and $P^2$-irreducible 3-manifold has positive first Betti number then its fundamental group has a left ordering, leaving $\bb Q$ homology 3-spheres as the interesting case for left orderings on 3-manifold groups \cite{BRW}. Moreover, the case of a Seifert fibred space was characterized in the same paper of Boyer, Rolfsen and Wiest. Given a hyperbolic knot $K$ it is often the case that $\dc K$ is hyperbolic, which makes these examples of interest. 

We remark on two interesting and related conjectures. Recall that an \textit{L-space} is a $\bb Q$ homology 3-sphere $M$ with the property that the total rank of the Heegaard Floer homology $\widehat{HF}(M)$ equals $|H_1(M)|$ (which is the smallest it can possibly be). It is a conjecture, raised by Ozsv\'ath and Szab\'o, that if an irreducible $\bb Z$ homology sphere is an L-space, it is $S^3$ or the Poincar\'e homology sphere (see e.g.\ \cite[Conjecture 1]{hedden-levine}). In \cite{BGW} it is conjectured that any irreducible $\bb Q$ homology sphere (not equal to $S^3$) is an L-space if and only if its fundamental group does not admit a left ordering. 

Combining the conjectures above we would expect, minus two exceptions, that irreducible $\bb Z$ homology spheres (and, in particular, the two-fold branched cover of a prime knot with determinant 1) will have a left-orderable fundamental group. Augmentations provide a sufficient condition. By Theorem \ref{thm-SL2R}, for any knot $K$ and a rank three, non-elliptic, real augmentation of $K$, $\rho$ induces an irreducible representation of $\pi_1(\dc K)$ into $\text{PSL}_2\bb R$, which implies $\pi_1(\dc K)$ has a circular ordering. Moreover the finite cover of $\dc K$ corresponding to the commutator subgroup has left-orderable fundamental group \cite[Proposition 1.2]{BRW}.

\begin{cor}If the determinant of $K$ is 1 and one of its rank three reflective augmentations is real and non-elliptic, then $\pi_1(\dc K)$ has a left-invariant order.
\label{cor-det1covers}
\end{cor}

In more generality, if the pullback by $\rho$ of the Euler class $\rho^*([e])\in H^2(\dc K)$ is zero, then it must be that $\pi_1(\dc K)$ has a left ordering (see \cite[Section 6.2]{ghys}); of course, when $\det(K)=1$ this is automatic as $H^2(\dc K)$ is trivial in this case. 

Utilizing a previous result in \cite{CH}, Corollary \ref{cor-det1covers} gives a method to construct infinite families of integer homology spheres which have left orderable fundamental group. For example, we may apply the following Corollary to any of the knots $K_0,K_1,K_2$ in Example \ref{example-3}. 

\begin{cor}Suppose that $P$ is a connected, closed braid contained in a solid torus $S^1\times D^2$ that is standardly embedded in $S^3$, meaning $S^1\times\{pt\}$ bounds a disk. For a knot $K\subset S^3$, let $P(K)$ be the satellite knot of $K$ with pattern $P$, and let $p$ be the winding number of $P$ in $S^1\times D^2$. If $p$ is odd and $K$ has a rank three reflective augmentation, which is real and non-elliptic, then so does $P(K)$. If $\det(K)=\det(P)=1$ as well, then $\dc{P(K)}$ has left orderable fundamental group.
\label{cor-braidsatellites}
\end{cor}
The proof will require some discussion of properties of a more general construction than the cord ring, the cord algebra $HC_0(K)$. Given an $n$-crossing diagram of $K$, $HC_0(K)$ admits a presentation similar to that of $\hc(K)$ in Theorem \ref{thm-knotdiagram}. However, in this case $c_{ij}\ne c_{ji}$ generally, for $1\le i\ne j\le n$. Please see \cite{Ng08} for details.

\begin{proof}Recall there is a correspondence, Theorem \ref{thm-correspondenceAugRep}, between elements in $\text{Hom}(HC_0(K),\bb C)$ and irreducible representations in $R_{\aug}(E_K)$. That correspondence relies on an isomorphism between $HC_0(K)$ and another algebra $\Pi_K$, over the ring $\bb Z[\mu^{\pm1},\lambda^{\pm1}]$, generated by the set of elements of $\pi_1(E_K)$. Given an element $\gamma\in\pi_1(E_K)$, write $[\gamma]$ for the corresponding element in $\Pi_K$. We note that the multiplications in $\pi_1(E_K)$ and $\Pi_K$ do not agree, i.e.\ $[\gamma_1\gamma_2]\ne[\gamma_1][\gamma_2]$. 

Let an $n$-crossing diagram of $K$ be chosen and set meridian generators to be $m_1,\ldots,m_n$ in $\pi_1(E_K)$. For each $1\le i\le n$ choose $\gamma_i$ such that $m_i = \gamma_i^{-1}m_1\gamma_i$ in $\pi_1(E_K)$. Define linking numbers $r_i=\text{lk}(\gamma_i,K)$. An isomorphism $HC_0(K)\cong\Pi_K$ can be chosen that sends $c_{ij}$ to $\mu^{r_j-r_i}[\gamma_i\gamma_j^{-1}]$, for $1\le i\ne j\le n$.

Given $k>0$, it was shown in \cite[Proposition 5.2]{CH} how to construct an irreducible element $\wt\rho\in R_{\aug}(E_{P(K)},k)$ from an irreducible $\rho\in R_{\aug}(E_K,k)$. In that proposition a collection of $np$ meridian generators $\{m_{ij}, 1\le i\le n, 1\le j\le p\}$ of $\pi_1(E_{P(K)})$ were chosen, so that $m_{i1}m_{i2}\ldots m_{ip}$ was identified with the meridian $m_i$ of $K$.

An augmentation $\ve\in\text{Hom}(HC_0(K),\bb C)$ factors through $\hc(K)$ (i.e.\ is reflective) if and only if $\ve(\mu)=-1$ and $\ve(c_{ij})=\ve(c_{ji})$. If $\ve$ is the augmentation corresponding to $\rho$ and $\ve(\mu) = -1$, then we remark that, since $(-1)^p=-1$ by assumption, the construction in \cite{CH} can be made so that $\wt\rho(m_{ik}) = \rho(m_i)$ for each $1\le i\le n$ and $1\le k\le p$. This implies $\wt\rho(\gamma_{ik}\gamma_{jl}^{-1}) = \rho(\gamma_i\gamma_j^{-1})$ which means that if $\ve$ is reflective then $\wt\ve$ is reflective (where $\wt\ve$ corresponds to $\wt\rho$ under Theorem \ref{thm-correspondenceAugRep}), and that $\wt\ve$ is elliptic or not according to $\ve$.

Finally, we note that $\D_{P(K)}(t) = \D_P(t)\D_{K}(t^p)$, and so $\det(P(K)) = 1$ if $\det(P)=1=\det(K)$ implying that $\dc{P(K)}$ has left orderable fundamental group by Corollary \ref{cor-det1covers}.
\end{proof}

We remark that one can construct a surjective homomorphism $\phi\co\pi_1(E_{P(K)})\to\pi_1(E_P)$ that preserves peripheral structure. In contrast, such a homomorphism does not generally exist from $\pi_1(E_{P(K)})$ onto $\pi_1(E_K)$ (see e.g.\ \cite[Proposition 3.4, Remark 3.5]{SW}). Taking advantage of the correspondence of Theorem \ref{thm-correspondenceAugRep}, one could use $\phi$ and a rank three augmentation $\ve\in\text{Hom}(HC_0(P),\bb C)$ to obtain a rank three augmentation $\wt\ve\in\text{Hom}(HC_0(P(K)),\bb C)$. Some care would be needed to check if $\ve$ being reflective implies $\wt\ve$ is reflective.

Recall Example \ref{example-2} above; let $K_0$ be the knot $10_{153}$. Since $y_3>2$, the augmentation $\ve_2$ we gave is non-elliptic and so determines a non-abelian representation $\rho\co\pi_1(\dc{K_0})\to\SL_2\bb R$. As $|\D_{K_0}(-1)|=1$, we have that $\pi_1(\dc{K_0})$ has a left ordering by Corollary \ref{cor-det1covers}.

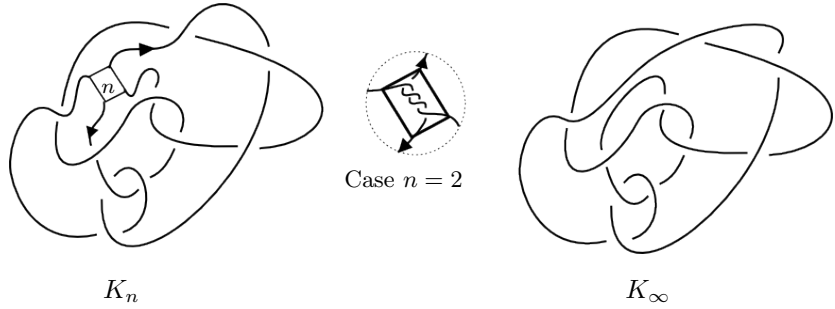
\begin{figure}[ht]
\begin{center}
\begin{tikzpicture}

\node[anchor=south west,inner sep=0] at (0,1) {\includegraphics[scale=0.5]{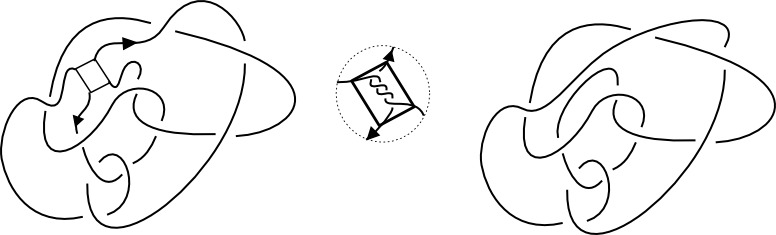}};

\draw(1.5,0.5) node {$K_n$};
\draw(8.5,0.5) node {$K_\infty$};
\draw(5.25,2) node {{\small Case $n=2$}};

\draw(1.325,3.2) node {{\footnotesize $n$}};

\end{tikzpicture}
\end{center}
\caption{Left: A family of knots $K_n$ with determinant 1.\ Right: The split link $K_\infty$.}
\label{fig-knotfamily}
\end{figure}

\begin{example} We consider a family of examples of homology 3-spheres for which $\dc{10_{153}}$ is the first. For each $n\in\bb Z$ we let $K_n$ be the knot with diagram shown on the left of Figure \ref{fig-knotfamily}, where the boxed region contains $n$ full-twists as shown (with the twists being in the opposite sense if $n<0$). Since we use full-twists, $K_n$ is connected and oriented as shown in the figure. Suppose that in the given diagram of $K_n$ the component at the top-left corner of the $n$-twist region is labeled $i_n$ and the component at the top-right corner is labeled $j_n$.

For all $n\in\bb Z$, $|\D_{K_n}(-1)|=1$. To see this, use the Conway-Alexander polynomial to see that $\D_{K_n}(t) = \D_{K_0}(t)$ for all $n$. Indeed, the difference between the Conway-Alexander polynomial of $K_n$ and that of $K_{n-1}$ is a multiple of the Conway-Alexander polynomial of the link $K_\infty$ on the right of Figure \ref{fig-knotfamily}. But $K_\infty$ is a split link (in fact it is isotopic to the 2-component unlink) and so has Alexander polynomial 0. Thus, since $|\D_{K_0}(-1)| = 1$, the determinant of $K_n$ is 1 for all $n\in\bb Z$.

We have already given a rank three non-elliptic augmentation for $K_0=10_{153}$ which implies $\pi_1(\dc{K_0})$ is left-orderable. Such an augmentation exists for $K_1$ and $K_2$ also. In the case of $K_1$, let $z_1$ is the largest root of $x^6+x^5-7x^4-2x^3+7x^2+2x-1$ (all of which are real). Then there is such an augmentation with $\ve_{i_1j_1} = z_1$, and such that all other coordinates of the augmentation are given by a polynomial in $\bb Z[z_1]$. For the case of $K_2$, let $z_2$ be the largest real root of the polynomial $x^7+x^6-4x^5-10x^4+4x^3+14x^2+4x^2-1$. Then there is a rank three non-elliptic augmentation of $K_2$ with $\ve_{i_2j_2} = z_2$ and with other coordinates given by a polynomial in $\bb Z[z_2]$. Both $z_1$ and $z_2$ are larger than 2, so these augmentations give an $\SL_2\bb R$ representation. Hence $\dc{K_1}$ and $\dc{K_2}$ have left-orderable fundamental group. 

We note that, although $\dc{K_{-1}}$ is expected to not be an L-space, $K_{-1}$ does not have a non-elliptic rank three augmentation (and so $\pi_1(\dc{K_{-1}})$ does not have a non-abelian equivariant $\SL_2\bb R$ representation). The variety of augmentations for our given diagram of $K_{-1}$ has 11 points. Other than the rank 1 point they are determined by setting $\ve_{i_{-1},j_{-1}}$ to a root of $1+2x-5x^2-18x^3-14x^4+14x^5+21x^6-3x^7-8x^8+x^{10}$. Four of these are real roots in $[-2,2]$, and the augmentations determined are elliptic. The other roots are not real, and so do not give a real-valued augmentation.
\label{example-3}
\end{example}

\clearpage
\appendix
\begin{center}{\bf Appendix.}
\end{center}

In the tables of this appendix we record the number of real, rank three, reflective augmentations, both elliptic and non-elliptic, for all prime 3-bridge knots with up to 10 crossings (no knots in this range have bridge index larger than 3). Among $8$-crossing and $9$-crossing knots, none has any non-elliptic augmentation, so we record only the number of elliptic augmentations.

For each knot the crossing information for a diagram was obtained from the Gauss notation given at KnotInfo \cite{knotinfo-site}. Solutions to equations (\ref{eqn-diagram-defneqns}) in Section \ref{sec-HCknotdiagram} were mostly found using numerical methods. In the cases that the variety has positive dimension, we restricted to a subset that is 1-dimensional and verified that elliptic augmentations do exist on that subset, as do non-elliptics.

{\footnotesize

\begin{table}[ht]
\centering
\begin{tabular}{|c|c|c|c|}
\hline
{Name} & {\begin{tabular}[c]{@{}c@{}}Elliptic\end{tabular}} & {Name} & \multicolumn{1}{c|}{\begin{tabular}[c]{@{}c@{}}Elliptic\end{tabular}}  \\ \hline
$8_5$          	& 1             & $9_{33}$		& 0                 \\ \hline 
$8_{10}$       	& 1             & $9_{34}$		& 1 				\\ \hline 
$8_{15}$       	& 1             & $9_{35}$		& 1 				\\ \hline 
$8_{16}$		& 0 		    & $9_{36}$		& 2 				\\ \hline
$8_{17}$		& 0 			& $9_{37}$		& 1 				\\ \hline
$8_{18}$		& 0 			& $9_{38}$		& 0 				\\ \hline	
$8_{19}$		& 1 			& $9_{39}$		& 2 				\\ \hline
$8_{20}$		& 1 			& $9_{40}$		& 3 				\\ \hline
$8_{21}$		& 1 			& $9_{41}$		& 2 				\\ \hline
$9_{16}$       	& 1  			& $9_{42}$		& 2 				\\ \hline
$9_{22}$       	& 2             & $9_{43}$		& 2 				\\ \hline   
$9_{24}$       	& 1             & $9_{44}$		& 2 				\\ \hline   
$9_{25}$		& 2 			& $9_{45}$ 		& 2 				\\ \hline	
 $9_{28}$		& 1 			& $9_{46}$ 		& 1 				\\ \hline
 $9_{29}$		& 0 			& $9_{47}$ 		& 1 				\\ \hline
 $9_{30}$		& 2 			& $9_{48}$		& 1 				\\ \hline
 $9_{32}$		& 0				& $9_{49}$ 		& 1 				\\ \hline
\end{tabular}
\caption{$8$-crossing and $9$-crossing knots with bridge number $3$}
\label{table-8crossing}
\end{table}

\begin{table}[ht]
\centering
\begin{tabular}{|c|c|c|c|c|c|}
\hline
{Name} & {\begin{tabular}[c]{@{}c@{}}Elliptic\end{tabular}} & \multicolumn{1}{c|}{\begin{tabular}[c]{@{}c@{}}Non-elliptic\end{tabular}} & {Name} & {\begin{tabular}[c]{@{}c@{}}Elliptic\end{tabular}} & \multicolumn{1}{c|}{\begin{tabular}[c]{@{}c@{}}Non-elliptic\end{tabular}}   \\ \hline	
$10_{124}$		& 2 				& 0		&$10_{145}$		& 3 				& 1		\\ \hline
$10_{125}$		& 2 				& 0		&$10_{146}$		& 3 				& 1		\\ \hline
$10_{126}$		& 2 				& 0		&$10_{147}$		& 3 				& 1		\\ \hline
$10_{127}$		& 2 				& 0		&$10_{148}$		& 2 				& 0		\\ \hline
$10_{128}$		& 2 				& 1		&$10_{149}$		& 2 				& 0		\\ \hline
$10_{129}$		& 2 				& 1		&$10_{150}$		& 2 				& 1		\\ \hline
$10_{130}$		& 2 				& 1		&$10_{151}$		& 2 				& 1		\\ \hline
$10_{131}$		& 2 				& 1		&$10_{152}$		& 4 				& 0		\\ \hline
$10_{132}$		& 2 				& 1		&$10_{153}$		& 4 				& 1		\\ \hline
$10_{133}$		& 2 				& 1		&$10_{154}$		& 4 				& 2		\\ \hline
$10_{134}$		& 2 				& 1		&$10_{155}$		& 2 				& 0		\\ \hline
$10_{135}$		& 2 				& 1		&$10_{156}$		& 3 				& 0		\\ \hline
$10_{136}$		& 3 				& 1		&$10_{157}$		& 3 				& 0		\\ \hline
$10_{137}$		& 3 				& 1		&$10_{158}$		& 2 				& 0		\\ \hline
$10_{138}$		& 3 				& 1		&$10_{159}$		& 3 				& 0		\\ \hline
$10_{139}$		& 2 				& 1		&$10_{160}$		& 3 				& 0		\\ \hline
$10_{140}$		& 2 				& 1		&$10_{161}$		& 2 				& 0		\\ \hline
$10_{141}$		& 2 				& 1		&$10_{162}$		& 1 				& 0		\\ \hline
$10_{142}$		& 2 				& 1		&$10_{163}$		& 2 				& 0		\\ \hline
$10_{143}$		& 2 				& 1		&$10_{164}$		& 4 				& 0		\\ \hline
$10_{144}$		& 2 				& 1		&$10_{165}$		& 2 				& 0		\\ \hline
\end{tabular}
\caption{$10$-crossing non-alternating knots with bridge number $3$}
\label{table-10crossingnon}
\end{table}

\begin{table}[ht]
\centering
\begin{tabular}{|c|c|c|c|c|c|}
\hline
{Name} & {\begin{tabular}[c]{@{}c@{}}Elliptic\end{tabular}} & \multicolumn{1}{c|}{\begin{tabular}[c]{@{}c@{}}Non-elliptic\end{tabular}} & {Name} & \multicolumn{1}{c|}{\begin{tabular}[c]{@{}c@{}}Elliptic\end{tabular}} & \multicolumn{1}{c|}{\begin{tabular}[c]{@{}c@{}}Non-elliptic\end{tabular}}  \\ \hline	
$10_{46}$		& 2 				& 0		& $10_{85}$		& 1 				& 0		\\ \hline
$10_{47}$		& 2 				& 0		& $10_{86}$		& 0 				& 0		\\ \hline %
$10_{48}$		& 2 				& 0		& $10_{87}$		& 1 				& 0		\\ \hline
$10_{49}$		& 2 				& 0		& $10_{88}$		& 0 				& 0		\\ \hline
$10_{50}$		& 2 				& 1		& $10_{89}$		& 0 				& 0		\\ \hline
$10_{51}$		& 2 				& 1		& $10_{90}$		& 2 				& 0		\\ \hline
$10_{52}$		& 2 				& 1		& $10_{91}$		& 0 				& 0		\\ \hline
$10_{53}$		& 2 				& 1		& $10_{92}$		& 0 				& 0		\\ \hline
$10_{54}$		& 2 				& 1		& $10_{93}$		& 2 				& 0		\\ \hline
$10_{55}$		& 2 				& 1		& $10_{94}$		& 0 				& 0		\\ \hline
$10_{56}$		& 2 				& 1		& $10_{95}$		& 0 				& 0		\\ \hline
$10_{57}$		& 2 				& 1		& $10_{96}$		& 0 				& 0		\\ \hline
$10_{58}$		& 3 				& 1		& $10_{97}$		& 0				& 0		\\ \hline
$10_{59}$		& 3 				& 1		& $10_{98}$		& $\ge1\ \dim$	& $\ge1\ \dim$	\\ \hline
$10_{60}$		& 3 				& 1		& $10_{99}$		& $\ge1\ \dim$	& $\ge1\ \dim$	\\ \hline
$10_{61}$		& 2 				& 1		& $10_{100}$		& 2 				& 0		\\ \hline
$10_{62}$		& 2 				& 1		& $10_{101}$		& 3 				& 0		\\ \hline
$10_{63}$		& 2 				& 1		& $10_{102}$		& 2 				& 1		\\ \hline
$10_{64}$		& 2 				& 1		& $10_{103}$		& 3 				& 0		\\ \hline
$10_{65}$		& 2 				& 1		& $10_{104}$		& 3 				& 0		\\ \hline
$10_{66}$		& 2 				& 1		& $10_{105}$		& 4 				& 0		\\ \hline
$10_{67}$		& 3 				& 1		& $10_{106}$		& 3 				& 0		\\ \hline
$10_{68}$		& 3 				& 1		& $10_{107}$		& 2 				& 2		\\ \hline
$10_{69}$		& 3 				& 1		& $10_{108}$		& 3 				& 0		\\ \hline
$10_{70}$		& 2 				& 0		& $10_{109}$		& 0 				& 2		\\ \hline
$10_{71}$		& 2 				& 0		& $10_{110}$		& 3 				& 1		\\ \hline
$10_{72}$		& 2 				& 0		& $10_{111}$		& 2 				& 0		\\ \hline
$10_{73}$		& 2 				& 0		& $10_{112}$		& 3 				& 0		\\ \hline
$10_{74}$		& 1 				& 0		& $10_{113}$		& 3 				& 0		\\ \hline
$10_{75}$		& 1 				& 0		& $10_{114}$		& 2 				& 0		\\ \hline
$10_{76}$		& 1 				& 0		& $10_{115}$		& 2 				& 0		\\ \hline
$10_{77}$		& 1 				& 0		& $10_{116}$		& 1 				& 0		\\ \hline
$10_{78}$		& 1 				& 0		& $10_{117}$		& 3 				& 0		\\ \hline
$10_{79}$		& 4 				& 0		& $10_{118}$		& 2 				& 0		\\ \hline
$10_{80}$		& 4 				& 1		& $10_{119}$		& 4 				& 0		\\ \hline
$10_{81}$		& 4 				& 2		& $10_{120}$		& 3 				& 0		\\ \hline
$10_{82}$		& 1 				& 0		& $10_{121}$		& 3 				& 1		\\ \hline
$10_{83}$		& 0 				& 0		& $10_{122}$		& 3 				& 0		\\ \hline
$10_{84}$		& 1 				& 0		& $10_{123}$		& $\ge1\ \dim$ & $\ge1\ \dim$	\\ \hline
\end{tabular}
\caption{$10$-crossing alternating knots with bridge number $3$}
\label{table-10crossingalt}
\end{table}

}

\clearpage
\bibliography{CharactersDoubleCover_refs}
\bibliographystyle{amsalpha-abbrv}

\end{document}